\documentclass[a4paper,11pt, DIV=10]{scrartcl}
\setcapindent{0pt}
\usepackage[style=numeric,maxbibnames=99,giveninits,doi=false,isbn=false,sorting=nyt]{biblatex}

\AtBeginBibliography{\small}

\usepackage[T1]{fontenc}
\usepackage{amsmath}
\usepackage{amssymb}
\addtokomafont{disposition}{\rmfamily\normalfont\scshape}
\usepackage{mathtools}
\usepackage[osf,sc]{mathpazo}
\usepackage{relsize}
\usepackage{amsthm}
\usepackage{enumerate}
\usepackage{algorithm}
\usepackage{algpseudocode}
\usepackage[autostyle=true]{csquotes}
\usepackage{verbatim}
\usepackage{graphicx}
\usepackage{placeins}
\usepackage{tikz}
\usepackage{tikz-cd}
\usepackage{microtype}
\usepackage{pdfpages}
\usepackage[british]{babel}
\usepackage[hidelinks,unicode]{hyperref}
\usepackage[commandnameprefix=ifneeded]{changes}
\usepackage{cleveref}
\crefname{algorithm}{Algorithm}{Algorithms}
\recalctypearea

\numberwithin{equation}{section}
\newtheorem{theorem}{Theorem}[section]
\newtheorem*{theorem*}{Theorem}
\newtheorem{proposition}[theorem]{Proposition}
\newtheorem{corollary}[theorem]{Corollary}
\newtheorem{lemma}[theorem]{Lemma}

\theoremstyle{definition}
\newtheorem{definition}[theorem]{Definition}
\newtheorem*{definition*}{Definition}

\theoremstyle{remark}

\newcommand{\Z}{\mathbb{Z}}

\newcommand{\R}{\mathbb{R}}

\newcommand{\mc}[1]{\mathcal{#1}}

\DeclareMathOperator{\diag}{diag}
\DeclareMathOperator{\vol}{vol}

\begin{document}
	\title{Selecting Interpretable Circular Coordinates from Data}
	\author{%
		Vincent P.\ Grande\\
		\textsc{rwth} Aachen University, Germany
		\and
		Marina Meil\u{a}\\
		University of Waterloo, Canada
	}
	\date{}
	\maketitle
	\begin{abstract}
		Circular coordinates obtained from persistent cohomology reveal loop structure in data, but they usually remain abstract: A detected circle does not tell us which measured angle, phase, torsion, or decoder explains it.
		We propose a method for selecting interpretable circle-valued coordinates from a user-supplied dictionary of scientifically meaningful candidates explaining the detected cohomology.
		In the continuous setting, each candidate is represented by the cohomology class of its pulled-back angular form, and selecting a minimum-energy set of candidates spanning the relevant \(H^1\) subspace becomes a minimum-weight basis problem in a vector matroid.
		We then introduce \textsc{circol}, a method for discrete point clouds sampled from the manifold.
		We prove that the introduced cochain inner product is a consistent estimator of the \(L^2\) inner product of fixed smooth \(1\)-forms under non-uniform sampling.
		The resulting projection matrix both helps selecting a basis of low-energy dictionary coordinates and diagnoses topologically trivial candidates or unexplained persistent classes.
		Finally, we verify the effectiveness of our method on synthetic examples, on molecular simulations, and neural recordings of head-direction cells.
	\end{abstract}
	\section{Introduction}
	In many real-world and experimental data sets, the quantity we care
about is cyclic, such as an angular direction, or a phase.  Yet each
individual measurement only indirectly records part of this circular
structure.

In single-cell gene expression data, cell-cycle phase and other
recurrent biological programs are recovered from thousands of gene
expression or multimodal measurements, mixed with
differentiation, noise, and sampling
effects \cite{vandereyken2023methods,cheng2024phlower}.  In molecular
dynamics, conformational change may be organised by rotations around
bonds, even though the recorded state is a large vector of atomic
coordinates \cite{Kovacev-Nikolic:2016}.  In neural recordings from
the brain, each observation is a high-dimensional snapshot of activity
across many neurons, and one subpopulation or circuit may encode an
angular variable such as head direction \cite{taube1990headI,duszkiewicz2024local,duszkiewicz2025dandiHeadDirection}.
In all of these examples, we care not only for checking whether circular structure exists \emph{somewhere} in data, but we want to explain the newly found topological structure in terms of known and interpretable quantities of our data.

Persistent (co-)homology can detect the existence of such loops.
The construction of circular coordinates from persistent
classes in $H^1$ replaces integer cocycles by harmonic representatives, and
integrates them modulo $\Z$ to obtain maps to
$\R/\Z$ \cite{DeSilva2009persistent}.
Later work has made these
coordinates sparser \cite{Perea2020}, more density-robust \cite{paik2023circular}, and better suited to
multiple circular
factors \cite{scoccola2023toroidal}.
The output, however, is still an abstract topological coordinate.  On
a neural population-activity point cloud, for example, this circular coordinate may reveal a
loop without saying whether that loop reflects behavioural head
direction, movement direction, the cyclic order of tuned neurons, or
some unrelated source of variation.  The same ambiguity appears in
other settings: a loop in molecular conformations could reflect one of
several torsion angles, while a loop in motion-capture or robotic
sensor data could reflect gait phase, joint rotation, or a repeated
task cycle.

The task of establishing the interpretation of each loop is performed by the domain scientist, typically by ad-hoc methods, such as visualizing each candidate domain-specific coordinate in turn, see for example \cref{fig:ethanol-candidates}.
Visualization is limited to low dimensional data embeddings and simple topologies, is often subjective and always time-consuming. 

This paper introduces \textsc{circol}, the first method to determine
which domain-specific variable, if any, explains a persistent loop, or
more generally the $1$-dimensional topology detected in the data.  The
scientist provides the list of candidate circular variables, which we
call the \emph{dictionary}, by analogy
with \cite{koelle2022manifold}; \textsc{circol} (Circular
Interpretability via Representative
COhomoLogy; \cref{alg:circular-dictionary-selection}) assigns to each persistent loop a (linear combination of) variable(s) in the dictionary.

The dictionary, i.e.\ the list of
candidate circle-valued functions, can include directly available measurements, like torsion angles in molecular dynamics,
fitted decoders, like in neuronal data, or domain-specific constructions, like cell cycle phase modules.

\paragraph{Contributions and outline}
We formulate dictionary selection for circular coordinates as a cohomological problem: each candidate angle is represented by the class of its pulled-back angular form, and selected candidates must span the chosen $H^1$ subspace.
We equip this problem with a Dirichlet-energy cost and show that the continuous selection problem is a minimum-weight basis problem in a vector matroid, solvable by the greedy algorithm.
In the finite sample point clouds, we design an algorithm, \textsc{circol}, that uses persistent cohomology, integer cocycle lifts, discrete harmonic representatives, and a density-corrected inner product on $1$-cochains.
We show that this inner product on $1$-cochains converges to the continuous inner product on differential forms in the infinite-sample limit.
The resulting projection matrix serves both selection and diagnostics, indicating which candidates explain the selected classes, which are topologically trivial, and which selected classes remain unexplained.

In \Cref{sec:circular-coordinates-manifolds} we discuss the problem in the continuous case.
\cref{sec:discrete-setting} is the core part of the paper discussing the finite-sample case and the convergence of the weighted inner product.
In \cref{sec:experiments} we test \textsc{circol} on synthetic data, on data from molecular simulation and on neural activity data.
Finally, we summarise and discuss limitations and future work in \cref{sec:discussion}, and give the full consistency proof in \cref{sec:circular-convergence}.

\subsection{Related Work}
The method to obtain circular coordinates from persistent cohomology was introduced by de Silva, Morozov, and Vejdemo-Johansson in \cite{DeSilva2009persistent}.
This was then extended to sparse circular coordinates using principal $\Z$-bundles and landmark sets \cite{Perea2020}, spherical coordinates obtained from $H^2$ \cite{schonsheck2024spherical}, and more generally Eilenberg-MacLane coordinates \cite{perea2017multiscaleProjective}.
For time-series data, sliding-window persistent homology provides a related way to quantify periodicity from point clouds built by delay embeddings \cite{perea2015sliding}.
One can view standard approaches to manifold learning like Laplacian eigenmaps \cite{belkin2003laplacian,coifman2006diffusion} as being obtained from integrating gradient eigenvectors of $1$-Hodge Laplacians, directly relating to the circular coordinate viewpoint of integrating harmonic eigenvectors, i.e.\ the $1$-forms associated with the $H^1$ classes.
\textsc{circol} does not introduce a novel circular coordinate construction, but rather determines which supplied circular dictionary elements represent the detected classes.

These coordinate constructions also contain an implicit or explicit variational step: after a cohomology class has been found, one chooses a low-energy or harmonic representative of that class.
In the continuous setting, harmonic maps are critical points of Dirichlet energy, and circle-valued maps associated with integral $1$-cohomology classes are governed by harmonic $1$-forms with integer periods \cite{baird2003harmonic,helein2008harmonic}.
For discrete point clouds, \cite{paik2023circular} adapted these minimisation techniques to non-uniform weights and introduce density-robust circular coordinates.
However, they formulate this in terms of the graph Laplacian, whereas we explicitly formulate the inner product on $1$-cochains and show convergence to differential forms in the infinite-sample limit.

For several independent $H^1$ classes, \cite{scoccola2023toroidal} deal with the problem of constructing an optimal basis and formulate this as a Dirichlet energy minimisation problem.
Their work is close to ours in its use of energy and integral changes of cohomology basis, but its goal is to find the best basis with integer periods of a given cohomology space, whereas our goal is to select and diagnose coordinates from a fixed external dictionary.

Our dictionary viewpoint is closest to ManifoldLasso and
TSLasso \cite{koelle2022manifold,koelle2024consistency}.
ManifoldLasso starts from an abstract manifold-learning embedding,
such as diffusion-map coordinates, and explains its coordinate
functions by sparse regression of their manifold gradients on
gradients of scientist-provided dictionary functions.  TSLasso removes
the embedding step and instead selects dictionary functions whose
gradients span the estimated tangent spaces of the manifold, giving an
interpretable parametrisation with finite-sample recovery guarantees.
Like ManifoldLasso and TSLasso, \textsc{circol} uses
scientist-provided dictionary functions to interpret structure learned
from data.
However, because of the additional algebraic structure of our problem, the solution entails mathematical and algorithmic concepts fundamentally different from those of ManifoldLasso and TSLasso. In particular, while the former algorithms find {\em local bases} of dictionary functions by a {\em standard sparse linear regression algorithm}, \textsc{circol} {\em pulls back} the dictionary functions in a {\em global basis}, where they have {\em integer coefficients}; moreover
the final optimisation is a {\em greedy minimum-weight basis problem in a
vector matroid} rather than a sparse convex regression problem.

There is also a growing literature using cohomological and Hodge-theoretic structure to interpret biological data.
Hodge decompositions have been used to analyse latent flows, RNA velocity and cell differentiation in single-cell data \cite{maehara2019modeling,su2024hodge,cheng2024phlower}.
In neuroscience, cohomological feature extraction has been used to decode head direction from mouse population activity \cite{rybakken2019decoding}, and persistent cohomology has been used to identify toroidal topology in grid-cell population activity \cite{gardner2022toroidal}.
Topological point features can be obtained by projecting selected persistent homology generators to harmonic representatives and pooling simplex weights to incident points, yielding interpretable local signals that indicate how individual points participate in global persistent classes \cite{grande2025point}.
Related work analyses small non-zero Hodge-Laplacian eigenvalues along $\alpha$-filtrations by separating harmonic, gradient, and curl modes \cite{grandeNotAllSmallEigenvalues}.
Maggs et al.\ recently introduced a cohomology-based framework for detecting and separating concurrent cyclic processes in single-cell transcriptomics \cite{Maggs2025TopologyConcurrentCyclicProcesses}.
The latter is particularly close in spirit because it searches biologically meaningful gene-set views for circular cohomological signal.
 Instead of discovering a gene subspace with strong cyclic signal, we audit a separate set of constructed circular candidates to explain the detected class.

Instead of representing forms in simplex-wise cochain bases, spectral exterior calculus and the diffusion-geometry framework build finite spectral or operator-theoretic representations from a small number of Laplacian eigenfunctions or sampled diffusion operators to approximate Hodge-type operators and differential forms \cite{berry2020spectral,jones2024diffusion,jones2024manifold}.
Compared with simplicial cochain methods, these approaches trade sparse local simplex-incidence structure, an exact finite cochain complex, exact harmonic representatives of finite cohomology classes, and integral cocycle data for compact spectral approximations with asymptotic convergence guarantees.

	\section{Circular coordinates on manifolds}
	\label{sec:circular-coordinates-manifolds}
	We will first develop a notion of continuous circular valued coordinates on manifolds using differential forms.
	We consider a smooth compact manifold $\mc{M}$ with Riemannian metric $g$.
	\paragraph{Singular cohomology}
	The first cohomology group $H^1(\mc{M};\Z)$ with $\Z$-coefficients describes the $1$-dimensional holes and loops of $\mc{M}$.
	Because $0$-dimensional homology is always free, the universal coefficient theorem states that $H^1(\mc{M};\Z)$ is torsion-free.
	Thus, the inclusion of coefficients $\iota\colon \Z\hookrightarrow \R$ induces an injection $\iota^*\colon H^1(\mc{M};\Z)\hookrightarrow H^1(\mc{M};\R)$.
	Furthermore, $\iota$ sends a basis of $H^1(\mc{M};\Z)$ to a basis of $H^1(\mc{M};\R)$.
	We can identify $\iota^*(H^1(\mc{M};\Z))$ with an integer lattice in $H^1(\mc{M};\R)$.

	\paragraph{Differential forms and de-Rham cohomology}
	The reason why we are interested in cohomology in real coefficients is the de-Rham theorem.
	The de-Rham theorem gives an explicit isomorphism between the de-Rham cohomology group $H^1_{dR}(\mc{M})$ on differential forms and $H^1(\mc{M};\R)$.
	Given a closed form $\alpha$ representing a class in $H^1_{dR}(\mc{M})$, we can evaluate it on a simplex $\sigma\colon \Delta^1\to \mc{M}$ by integrating $\alpha$ over the image of $\sigma$.
	This generates a map $k\colon \Omega^1(\mc{M})\to \hom_\R (C_1(\mc{M};\R),\R)$,
	which in turn induces an isomorphism in cohomology $[k]\colon H^1_{dR}(\mc{M})\to H^1(\mc{M};\R)$.

	While this is an isomorphism in cohomology, there is no inverse on the cochain level:
	Some cochains do not correspond to differential forms.
	The intuition for this fact is that while differential forms are continuous objects, cochains can be arbitrary locally complicated functions on simplices.

	Again, the integer cohomology classes now represent an integer lattice in the de-Rham cohomology $H^1_{dR}(\mc{M})$, corresponding to cohomology generated by differential $1$-forms, which are called differential forms with integer periods.
	
	We consider a cycle $c\colon S^1\to \mc{M}$, representing a homology class $[c]\in H_1(\mc{M};\Z)$.
	We can now integrate a differential form $\alpha$ over $c$ and equate this to the evaluation of the cochain $\bar{\alpha}$ on a representative $\bar{c}$ of the homology class $[c]$:
		\begin{equation}
		\int_c \alpha = \bar{\alpha}(\bar{c})\in \R.
		\end{equation}
	If $[\alpha]$ is an integer cohomology class, this integral is in $\Z$.
	By Stokes' theorem, this integral only depends on the homology class $[c]$.

	\paragraph{Continuous circular coordinates}
	Given a differential $1$-form $\alpha\in \Omega^1(\mc{M})$ and a point $x_0\in \mc{M}$, we can send every path $\gamma\colon [0,1]\to \mc{M}$ with $\gamma(0)=x_0$ to the integral
		\begin{equation}
		\int_{[0,1]} \gamma^*\alpha \mod \Z \in \R/\Z.
		\end{equation}
	If $\alpha$ represents an integer cohomology class, this is independent of the choice of the path $\gamma$ between $x_0$ and $\gamma(1)$, because cycles will evaluate to integers.
	Thus, we obtain a well-defined map $\theta\colon \mc{M}\to \R/\Z$.
	We call such a map $\theta$ a \emph{circular coordinate} on $\mc{M}$.
	
	In summary, given a smooth compact manifold $\mc{M}$, we can compute its first integer cohomology group $H^1(\mc{M};\Z)$.
	Choosing a basis of this group, we can find a corresponding basis of differential $1$-forms with integer periods.
	These forms then induce a system of  circular coordinates $\theta_i\colon \mc{M}\to \R/\Z$ on $\mc{M}$ that parametrise the non-trivial topology of $\mc{M}$.

	\paragraph{Coordinates with physical meaning}

	Now assume that we have a dictionary $\mathcal{G}=\{g_j\colon \mc{M}\to \R/\Z, j=1,\dots, J\}$ of candidate circular coordinates on $\mc{M}$.
	These coordinates could come from physical measurements, e.g.\ angles of joints in a robotic system, or phase angles in a dynamical system.
	We want to select a small subset of these coordinates that \enquote{explain} the topology of $\mc{M}$.
	
	But what does it mean to \enquote{explain} the topology of $\mc{M}$?
	We can reverse the construction above:
	Given a circular coordinate $g_j\colon \mc{M}\to \R/\Z$, we can associate to it a differential $1$-form $\omega (g_j)\in \Omega^1(\mc{M})$ by pulling back the standard angular form $d\theta$ on $\R/\Z$:
		\begin{equation}
		\omega (g_j) \coloneq g_j^* d\theta.
		\end{equation}
	This form is closed, i.e.\ $d\omega (g_j)=0$ because the exterior derivative commutes with pullbacks and $dd\theta=0$.
	Thus, $\omega (g_j)$ represents a cohomology class $[\omega (g_j)]\in H^1_{dR}(\mc{M})$ in the real de-Rham cohomology of $\mc{M}$.
	Hence, for a subset of circular coordinates to \enquote{explain} the topology of $\mc{M}$, we mean that the set of cohomology classes $[\omega (g_j)]$ for $j$ in this subset forms a basis of $H^1_{dR}(\mc{M})$.

	\paragraph{The Dirichlet energy of this basis}
	
	We now want to choose among all subsets of circular coordinates that explain the topology of $\mc{M}$ the ones that are in some sense the \enquote{smoothest}.
	This means that they should explain the topology of $\mc{M}$ with as little \enquote{variation} as possible over the entire manifold $\mc{M}$.
	Given a Riemannian metric $g$ on $\mc{M}$, we can consider the $L_2$-norm of differential forms induced by $g$.
	The existence of a well-defined inner product is another reason to work with deRham cohomology on differential forms instead of with simplices.
	Using the Hodge star operator $\star$, we can write this norm as
		\begin{equation}
		\lVert \alpha \rVert^2 = \int_{\mc{M}} \alpha \wedge \star \alpha.
		\end{equation}
	Equivalently,
	we can write this norm using the pointwise inner product on forms induced by the metric $g$:
		\begin{equation}
		\lVert \alpha \rVert^2 = \int_{\mc{M}} \langle \alpha, \alpha \rangle_g \, \vol_{\mc{M}},
		\end{equation}
	where $\vol_{\mc{M}}$ is the volume form on $\mc{M}$ induced by $g$.
	Given a circular coordinate $g_j$, we can compute the norm of the associated differential form $\omega (g_j)$ using this formula.
	This is the same as the Dirichlet energy of the map $g_j\colon \mc{M}\to \R/\Z$, or the $L_2$-norm of its gradient.
	Using this norm, we can define the Dirichlet energy of a set of circular coordinates $\{g_{j}\}_{j\in \mc{S}}$ as the sum of the squared norms of the associated differential forms:
		\begin{equation}
		E(\{g_{j}\}_{j\in \mc{S}}) \coloneq \sum_{j\in \mc{S}} \lVert \omega (g_j)\rVert^2.
		\end{equation}
	\paragraph{Continuous problem formulation}
	Given a smooth compact connected Riemannian manifold $\mc{M}$ with metric $g$ and a dictionary $\mathcal{G}=\{g_j\colon \mc{M}\to \R/\Z, j=1,\dots, J\}$ of candidate circular coordinates on $\mc{M}$, we want to find a subset $\mc{S}\subset [1,\dots, J]$ of circular coordinates such that
	\begin{enumerate}
		\item The set of cohomology classes $\{[\omega (g_j)]\}_{j\in \mc{S}}$ forms a basis of real-valued de-Rham cohomology $H^1_{dR}(\mc{M})$.
		\item The Dirichlet energy $E(\{g_{j}\}_{j\in \mc{S}})$ is minimal among all subsets satisfying (1).
	\end{enumerate}
\paragraph{Formulation as a matroid problem}
In this paragraph, we reformulate the above problem in terms of linear algebra.
Let $\{\alpha_1,\dots, \alpha_k\}$ be harmonic representatives of a basis of the de-Rham cohomology $H^1_{dR}(\mc{M})$ of $\mc{M}$.
These representatives need not be orthonormal with respect to the inner product induced by the Riemannian metric $g$.
We can now consider the Gram matrix $Q\in \R^{k\times k}$ defined as
\begin{equation}
Q_{ij} \coloneq \langle \alpha_i, \alpha_j \rangle
\end{equation}
where $\langle \cdot, \cdot \rangle$ is the inner product on differential forms induced by $g$.
Given a circular coordinate $g_j$, we can compute the projection of the associated differential form $\omega (g_j)$ onto the basis $\{\alpha_1,\dots, \alpha_k\}$ as
\begin{equation}
P(g_j) \coloneq Q^{-1} \begin{pmatrix}
\langle \alpha_1, \omega (g_j) \rangle \\
\vdots \\
\langle \alpha_k, \omega (g_j) \rangle
\end{pmatrix} \in \R^k.
\end{equation}
The same coefficients determine the harmonic representative $h_j$ of the cohomology class $[\omega(g_j)]$:
\begin{equation}
h_j = \sum_{i=1}^k P(g_j)_i \alpha_i.
\end{equation}
Therefore
\begin{equation}
P(g_j)^\top Q P(g_j)=\lVert h_j\rVert^2.
\end{equation}
If $\omega(g_j)=h_j+df_j$ is the Hodge decomposition of $\omega(g_j)$, then $\lVert \omega(g_j)\rVert^2=\lVert h_j\rVert^2+\lVert df_j\rVert^2$.
Using this notation, we can reformulate the problem above as follows:
Given vectors $v_j = P(g_j) \in \R^k$ with costs $c_j = \lVert \omega (g_j) \rVert^2$, pick a subset $\mc{S}\subset [1,\dots, J]$ minimising $\sum_{j\in \mc{S}} c_j$ such that the set of vectors $\{v_j\}_{j\in \mc{S}}$ forms a basis of $\R^k$ , i.e.\ such that the matrix $[v_j]_{j\in \mc{S}}$ has non-zero determinant.
This is a classical problem in matroid theory, called the \emph{minimum weight basis problem}, and can be solved using greedy algorithms in polynomial time.
\section{Discrete setting}
\label{sec:discrete-setting}
In practice, we do not have access to the full manifold $\mc{M}$, but only to a finite set of points $X=\{x_1,\dots, x_N\}\subset \mc{M}\subset \R^D$ sampled from $\mc{M}$ in $D$-dimensional ambient space.
The discrete version of \textsc{circol} therefore replaces the continuous objects from the previous section by cochains, persistent classes, and empirical inner products on a simplicial complex.
This means we are faced with new challenges:
\begin{enumerate}
	\item We need to estimate the cohomology $H^1_{dR}(\mc{M})$ from the point cloud $X$.
	\item We need to estimate the inner product on cochains as a discretization of differential forms on $\mc{M}$ from the point cloud $X$.
	\item We need to turn the sampled circular coordinates $g_j\colon X\to \R/\Z$ into real $1$-cocycles on $\mc{S}_\varepsilon$, resolving the modulo-$\Z$ ambiguity on edges so that their cohomology classes can be compared with the selected persistent classes.
\end{enumerate}
\paragraph{Persistent cohomology}
To estimate the cohomology of $\mc{M}$ from the point cloud $X$, we can use persistent cohomology:
First, we construct a filtered simplicial complex $\mc{S}_\varepsilon$ from $X$, e.g.\ a Vietoris--Rips filtration or an $\alpha$-filtration in low ambient dimension.

We then compute the persistent cohomology of this filtration with $\Z/pZ$-coefficients for a small odd prime $p$ in dimension $1$.
Using a combination of domain knowledge and heuristics, we can then select a set of representative cohomology classes $\alpha_k$ and a scale parameter $\varepsilon$ at which we consider the simplicial complex $\mc{S}_\varepsilon$ to be a good approximation of the manifold $\mc{M}$.

In theory, we do not need to pick a single scale $\varepsilon$, but can consider a different scale for each cohomology class.
However, we will now assume that there exists a single scale $\varepsilon$ at which all selected cohomology classes are \enquote{alive}.

Finally, following the lifting step in the original circular-coordinate construction \cite[Section~2.4]{DeSilva2009persistent}, we attempt to lift the selected cohomology classes from $\Z/pZ$-coefficients to integer coefficients $\Z$.
We take centered integer representatives of the selected $\Z/pZ$-cocycles and use the Bockstein obstruction to test whether their triangle defects can be removed by an integer cochain correction.
When this obstruction vanishes, the correction yields integer cohomology classes $[\beta_1]_\Z,\dots, [\beta_k]_\Z$ in $H^1(\mc{S}_\varepsilon;\Z)$.
\paragraph{Estimating the inner product on cochains}
The next step would now be to relate the candidate circular coordinates to the above cohomology classes.
For these projections, we however first need to understand how to compute the inner product on 1-cochains in the discrete setting, which we will derive in this paragraph
The argument to get from continuous $L_2$ inner product to our discrete cochain inner product has four steps:
\begin{enumerate}
\item rewrite the pointwise inner product as an average over directions
\item replace those infinitesimal directional evaluations by short geodesic integrals
\item integrate over the manifold and correct for non-uniform sampling
\item discretize the resulting pairwise formula on the edges of the simplicial complex.
\end{enumerate}

Let $E_1$ be a choice of orientation of the edges of $\mc{S}_\varepsilon$.
Let $\omega_{d-1}$ be the volume of the Euclidean unit sphere in dimension $d$.
Let $\mu$ denote Riemannian volume on a smooth compact $d$-dimensional Riemannian manifold $\mc{M}$ without boundary, take a bandwidth $h>0$, and define
\begin{equation}
K_h(x,y)=h^{-d}\kappa\!\left(\frac{d_{\mc{M}}(x,y)}{h}\right),
\end{equation}
where $\kappa\colon [0,\infty)\to[0,\infty)$ is bounded, supported in $[0,1)$, not identically zero, and has moment constants
\begin{equation}
m_0=\omega_{d-1}\int_0^\infty \kappa(\rho)\rho^{d-1}\,d\rho,
\qquad
m_2=\omega_{d-1}\int_0^\infty \kappa(\rho)\rho^{d+1}\,d\rho.
\end{equation}
Since $\kappa$ is non-negative and non-zero on a set of positive measure, $m_2>0$, so the normalization below is well-defined.
For the unweighted Vietoris--Rips neighbourhood kernel $\kappa_{\mathrm{VR}}(\rho)=\mathbf 1_{\rho<1}$, these constants can be computed as:
\begin{equation}
m_0=\frac{\omega_{d-1}}{d},
\qquad
m_2=\frac{\omega_{d-1}}{d+2},
\qquad
\frac{2dm_0^2}{m_2h^2}=\frac{2\omega_{d-1}(d+2)}{dh^2}.
\end{equation}
For data points $x_1,\dots,x_N$, define the empirical kernel masses
\begin{equation}
q_i=\sum_{\ell\neq i}K_h(x_i,x_\ell).
\end{equation}
For an oriented edge $e=(x_i,x_j)\in E_1$, set
\begin{equation}
w_e=
\begin{cases}
\dfrac{2dm_0^2}{m_2h^2}\dfrac{K_h(x_i,x_j)}{q_iq_j}, & \text{if } K_h(x_i,x_j)>0,\\[0.8ex]
0, & \text{if } K_h(x_i,x_j)=0,
\end{cases}
\end{equation}
and define, for $1$-cochains $c,c'\in C^1(\mc{S}_\varepsilon;\R)$,
\begin{equation}
\langle c,c'\rangle_M
=
\sum_{e\in E_1} w_e\,c(e)c'(e),
\qquad
M=\diag(w_e)_{e\in E_1}.
\end{equation}
The factors $q_iq_j$ remove the leading effect of non-uniform sampling density.
The factor $h^{-2}$ appears because a $1$-cochain stores an edge integral.
For a short edge of length $r$, such an integral is of order $r$, so the product of two cochain values already contains the required factor $r^2$ and we do not divide by edge lengths.
The normalization is chosen to estimate the $L_2$ inner product of differential $1$-forms from edge integrals.
The factor $2$ compensates for using one chosen orientation of each undirected edge, whereas the continuum double integral counts both orientations.

We will now reproduce the consistency theorem, for which we will give the proof in \cref{sec:circular-convergence}:
\begin{theorem}[Consistency of the cochain inner product]
\label{thm:cochain-inner-product-consistency-main-text}
Assume that $x_1,\dots,x_N$ are sampled i.i.d.\ from $\nu=\pi\mu$ with $\pi>0$ smooth, that $h=h_N\to 0$ and $Nh_N^d/\log N\to\infty$, and that the $1$-skeleton contains all pairs with $K_{h_N}(x_i,x_j)>0$.
For fixed smooth $1$-forms $\alpha,\beta$, let $c_\alpha(e)$ and $c_\beta(e)$ be their geodesic integrals along the oriented edge $e$.
Then
\begin{equation}
\label{eq:discrete-inner-product-consistency-main-text}
\langle c_\alpha,c_\beta\rangle_M
=
\langle \alpha,\beta\rangle_{L_2(\mu)}
+ O(h_N^2)
+ O_{\mathbb P}\!\left(\sqrt{\frac{\log N}{N h_N^d}}\right),
\end{equation}
and in particular $\langle c_\alpha,c_\beta\rangle_M\to \langle \alpha,\beta\rangle_{L_2(\mu)}$ in probability.
\end{theorem}
We will briefly give an intuition for the proof below:
First, the sphere-average identity rewrites $\langle \alpha_x,\beta_x\rangle_g$ as an average of $\alpha_x(\theta)\beta_x(\theta)$ over unit tangent directions.
In geodesic normal coordinates, we can do Taylor expansion to the edge integrals $\alpha(x,y)$ and $\beta(x,y)$ and observe that the odd cubic terms vanish against the radial kernel, leaving an $O(h^2)$ bias after the $h^{-2}$ normalization.
The factors $q_iq_j$ are discretizations of the continuum density correction $q_h(x)q_h(y)$, and we can control the stochastic part by concentration estimates for the empirical kernel masses and by a variance bound for the resulting pairwise average.

\paragraph{Computing discrete harmonic representatives}
Combining the selected integer cocycles with the cochain inner product, we compute real-valued harmonic representatives of the classes $[\beta_i]_\R$.
Concretely, for each $i$ we take the minimum-norm representative in the real cohomology class,
\begin{equation}
\alpha_i
\in
\arg\min_{\alpha\in \beta_i+\operatorname{im} B_1^\top}
\lVert \alpha\rVert_M^2.
\end{equation}
These cochains represent the selected persistent classes and approximate harmonic $1$-forms on $\mc{M}$.
We write $\alpha_1,\dots,\alpha_k\in C^1(\mc{S}_\varepsilon;\R)$ for these representatives and write $A$ for the matrix with these representatives as columns.

\paragraph{From candidate circular functions to $1$-cochains and harmonic projections}
We again consider a circular coordinate $g_j\colon \mc{M}\to \R/\Z$ from the dictionary $\mathcal{G}$.
This circular coordinate on the entirety of $\mc{M}$ introduces a coordinate function on the point cloud $X\to \R/\Z$ by restriction.
For every oriented edge $e=(x,y)$ in the simplicial complex $\mc{S}_\varepsilon$, we compute the short angular difference
\begin{equation}
\omega(g_j)(e)
\in (-1/2,1/2]
\end{equation}
as the unique real number satisfying
\begin{equation}
\omega(g_j)(e)
\equiv
g_j(y)-g_j(x)
\pmod{\Z}.
\end{equation}
This gives a real $1$-cochain $\omega(g_j)\in C^1(\mc{S}_\varepsilon;\R)$ that approximates the edge integral of $g_j^*d\theta$.
We denote by $G$ the matrix whose $j$-th column is the cochain $\omega(g_j)$ in the standard basis of $C^1(\mc{S}_\varepsilon;\R)$.

Let us denote the first and second boundary matrices of $\mc{S}_\varepsilon$ by $B_1$ and $B_2$, respectively.
With a consistent unwrapping on every $2$-simplex, the real triangle defects vanish and
\begin{equation}
d^1\omega(g_j)=B_2^\top \omega(g_j)=0.
\end{equation}
Under this condition, $\omega(g_j)$ is a cocycle and represents a possibly trivial cohomology class in $H^1(\mc{S}_\varepsilon;\R)$.
Non-zero integer triangle defects indicate that the sampled vertex values do not define a real cocycle on the chosen complex without further unwrapping or correction.

Using the inner product on cochains, we can associate a \enquote{cost} $c_j$ to the circular coordinate $g_j$ by computing the norm of $\omega (g_j)$.
Furthermore, we can relate $\omega (g_j)$ to the topology of $\mc{M}$ by projecting the $\omega (g_j)$ onto the subset spanned by the discrete harmonic representatives $A$ and express this projection in this basis.
However, we need to be careful, because the representatives $\alpha_i$ are not orthonormal with respect to the inner product on cochains.
We can compute the Gram matrix $Q$ of the representatives $\alpha_i$ as
\begin{equation}
Q_{ij} = \langle \alpha_i, \alpha_j \rangle_M = \alpha_i^\top M \alpha_j.
\end{equation}
Using this Gram matrix, we can now compute the projection of each $\omega(g_j)$ onto the space spanned by the discrete harmonic representatives as
\begin{equation}
p_j = Q^{-1}A^\top M \omega(g_j).
\end{equation}
\begin{equation}
P=Q^{-1} A^\top M G \in \R^{k\times J}.
\end{equation}
$P$ is now the matrix of the projections of the $\omega (g_j)$ onto the space spanned by the discrete harmonic representatives in the basis given by these representatives, and in particular a matrix of dimension $k\times J$, where $k$ is the number of selected cohomology classes and $J$ is the number of circular coordinates in the dictionary $\mathcal{G}$.

\paragraph{Dealing with \enquote{noisy} homology classes}
In geometric machine learning and topological data analysis, a huge emphasis is put on the question of selecting the \enquote{right} scale to reconstruct a manifold from a point cloud.
In practice, there is often no single scale at which the homology of the simplicial complex $\mc{S}_\varepsilon$ perfectly matches the homology of the underlying manifold $\mc{M}$.
Even when selecting homology classes using persistent homology across scales, there is usually not a heuristic that can perfectly separate the signal from the \enquote{noise}, that could for example arise from less sampled regions of the manifold.

In this work, we argue that we can circumvent this problem for our purposes by using the dictionary of circular coordinates $\mathcal{G}$ as a second criterion of \enquote{truthful reconstruction}:
The following theorem states that any \enquote{noisy} homology class that does not correspond to a true homology class of the underlying manifold $\mc{M}$ will not be \enquote{explained} by any valid circular coordinate in the dictionary $\mathcal{G}$, and thus can be excluded before the optimisation problem is solved.
\begin{theorem}
Let $(N,\mathrm g)$ be a closed oriented Riemannian manifold, $i\colon N\to M$ a smooth map, and
$h\in H_1(N;\mathbb R)$.  
Let $\alpha_h\in\mathcal H^1(N)$ be the harmonic $1$-form associated to $h$. 
For a smooth map $g\colon M\to S^1$ (the candidate circular coordinate), we define the differential form
\begin{equation}
\beta\coloneq(i^*g)^*(d\theta)=i^* \bigl(g^*(d\theta)\bigr)\in\Omega^1(N),
\end{equation}
where $d\theta$ is the standard closed $1$-form on $S^1=\mathbb R/\mathbb Z$.
If $i_*h=0$ in $H_1(M;\mathbb R)$, then
\begin{equation}
\langle \alpha_h,\beta\rangle_{L^2}=0.
\end{equation}
\end{theorem}

\begin{proof}
Since $\beta$ is closed, Hodge decomposition on $1$-forms gives
\begin{equation}
\beta=\omega+df,\qquad \omega\in\mathcal H^1(N),\ f\in C^\infty(N).
\end{equation}
Hence
\begin{equation}
\langle \alpha_h,\beta\rangle
=\langle \alpha_h,\omega\rangle+\langle \alpha_h,df\rangle.
\end{equation}
Because $\alpha_h$ is harmonic, $\delta\alpha_h=0$, so on closed $N$,
\begin{equation}
\langle \alpha_h,df\rangle=\langle \delta\alpha_h,f\rangle=0.
\end{equation}
Therefore
\begin{equation}
\langle \alpha_h,\beta\rangle=\langle \alpha_h,\omega\rangle=\int_h \omega=\int_h \beta,
\end{equation}
by the definition of associated harmonic class and since $\beta-\omega=df$ is exact.

Now set $\eta:=g^*(d\theta)\in\Omega^1(M)$. Then $\beta=i^*\eta$, so by naturality of pairing,
\begin{equation}
\int_h \beta=\int_h i^*\eta=\int_{i_*h}\eta.
\end{equation}
If $i_*h=0\in H_1(M;\mathbb R)$, pairing with any closed $1$-form vanishes; in particular
\begin{equation}
\int_{i_*h}\eta=0.
\end{equation}
Thus $\langle \alpha_h,\beta\rangle_{L^2}=0$.
\end{proof}
This shows that for every homology class $h$ that does not correspond to a true homology class of the underlying manifold $\mc{M}$, we have $\langle \alpha_h,\beta\rangle_{L^2}=0$ for all circular coordinates $g\in \mathcal{G}$.
In other words, the corresponding row of $A^\top MG$ will be zero, and we can directly exclude the homology class $h$ from the optimisation problem.
It could of course still happen that some homology classes that do correspond to true homology classes of $\mc{M}$ are not \enquote{explained} by any circular coordinate in the dictionary $\mathcal{G}$, but as we focus on the problem of best explanation of the topology of $\mc{M}$ by the given coordinates in the dictionary $\mathcal{G}$, this is not a problem for our purposes.
\paragraph{Discrete problem definition}
Given a point cloud $X=\{x_1,\dots, x_N\}\subset \R^D$ sampled from a smooth compact connected Riemannian manifold $\mc{M}$, a filtered simplicial complex $\mc{S}_\varepsilon$ built on top of $X$, a set of integer cohomology classes $[\beta_1]_\Z,\dots, [\beta_k]_\Z$ in $H^1(\mc{S}_\varepsilon;\Z)$, an inner product on cochains governed by matrix $M$, and a dictionary $\mathcal{G}=\{g_j\colon X\to \R/\Z, j=1,\dots, J\}$ of candidate circular coordinates on $X$, the finite-sample \textsc{circol} problem is to find a subset $\mc{S}\subset [1,\dots, J]$ of circular coordinates such that
\begin{enumerate}
	\item The set of projections $\{P(g_j)\}_{j\in \mc{S}}$ forms a basis of $\R^k$.
	\item The cost $\sum_{j\in \mc{S}} c_j$ is minimal among all subsets satisfying (1).
\end{enumerate}

\paragraph{Solving the optimization problem}
\begin{definition}[Matroid]
A matroid is a pair $(E,\mc{I})$ where $E$ is a finite set and $\mc{I}\subset 2^E$ is a non-empty collection of subsets of $E$, called independent sets, such that
\begin{enumerate}
	\item If $I\in \mc{I}$ and $I'\subset I$, then $I'\in \mc{I}$.
	\item If $I_1, I_2\in \mc{I}$ and $\lvert I_1 \rvert < \lvert I_2 \rvert$, then there exists an element $e\in I_2\setminus I_1$ such that $I_1 \cup \{e\}\in \mc{I}$.
\end{enumerate}
\end{definition}

\begin{theorem}[Matroid optimisation, \cite{edmonds1971matroids}]
Let $(E,\mc{I})$ be a matroid and $c\colon E\to \R_{\geq 0}$ a cost function on the elements of $E$.
Then, the problem of finding a maximal independent set $I\in \mc{I}$ minimising the cost $\sum_{e\in I} c(e)$ can be solved by a greedy algorithm that iteratively adds the element $e\in E$ with lowest cost $c(e)$ such that $I\cup \{e\}\in \mc{I}$.
\end{theorem}
Our optimisation problem can be reformulated as a matroid optimisation problem.
Here the relevant matroid is the vector matroid on the ground set $[1,\dots,J]$, where a subset $I\subset [1,\dots,J]$ is independent exactly when $\{p_j\}_{j\in I}$ is linearly independent.

Hence, we have gathered all ingredients we need for formulating our \textsc{circol} algorithm in \cref{alg:circular-dictionary-selection}.
We will conclude this section with a paragraph interpreting ordinary Laplacian eigenmaps and circular coordinates as having the same motive of being \emph{integrals over certain eigenvectors of the 1-Hodge Laplacian.}

\begin{algorithm}[tb]
\caption{The \textsc{circol} algorithm for selecting circular coordinates from a dictionary}
\label{alg:circular-dictionary-selection}
\begin{scriptsize}
	\begin{algorithmic}[1]
\Require Point cloud $X$; dictionary $\mathcal{G}=\{g_j\colon X\to \R/\Z\}_{j=1}^J$; filtered simplicial complex $(\mathcal{S}_\varepsilon)_\varepsilon$
\Ensure Subset $\mathcal{S}\subseteq \{1,\dots,J\}$ whose dictionary elements explain the selected topology with minimal total cost
\State Compute persistent cohomology of $(\mathcal{S}_\varepsilon)_\varepsilon$ in degree $1$
\State Choose a scale $\varepsilon$ and $k$ persistent classes alive at that scale
\State Construct $\mathcal{S}_\varepsilon$ and the cochain inner product matrix $M$ on $C^1(\mathcal{S}_\varepsilon;\R)$
\State Lift the selected classes to integer cocycles $\beta_1,\dots,\beta_k\in C^1(\mathcal{S}_\varepsilon;\Z)$
\State Compute discrete harmonic representatives $\alpha_1,\dots,\alpha_k\in C^1(\mathcal{S}_\varepsilon;\R)$
\State Set $A\gets[\alpha_1\ \cdots\ \alpha_k]$ and $Q\gets A^\top M A$
\For{$j=1,\dots,J$}
	\ForAll{oriented edges $e=(x,y)\in E_1$}
		\State Choose $\omega_j(e)\in(-1/2,1/2]$ with $\omega_j(e)\equiv g_j(y)-g_j(x)\pmod{\Z}$
	\EndFor
	\State Set $c_j\gets \omega_j^\top M\omega_j$
	\State Set $p_j\gets Q^{-1}A^\top M\omega_j$
\EndFor
\State Sort the indices $j$ by increasing cost $c_j$
\State Initialise $\mathcal{S}\gets\varnothing$
\ForAll{indices $j$ in sorted order}
	\If{$p_j$ is linearly independent of $\{p_\ell\}_{\ell\in\mathcal{S}}$}
		\State Set $\mathcal{S}\gets\mathcal{S}\cup\{j\}$
	\EndIf
	\If{$|\mathcal{S}|=k$}
		\State \textbf{break}
	\EndIf
\EndFor
\State \Return $\{g_j\}_{j\in\mathcal{S}}$
\end{algorithmic}
\end{scriptsize}

\end{algorithm}
\paragraph{Diffusion and circular coordinates as edge integrals}
Diffusion maps and Laplacian eigenmaps are closely related spectral embedding methods built from a weighted neighbourhood graph on the data \cite{coifman2006diffusion,belkin2003laplacian}.
Laplacian eigenmaps use low-frequency eigenvectors of a graph Laplacian as smooth real-valued coordinates on the data.
Diffusion maps usually normalise the same kernel graph to a Markov transition operator and scale the resulting eigenvectors according to a diffusion time.
Modulo these normalisation choices, both methods use eigenvectors of graph-Laplacian-type operators as real-valued coordinates.
For the present comparison we will use the unnormalised graph-Laplacian model, but note that the same structure carries on to weighted Laplacians.
Let $\Gamma$ be the graph given by the $1$-skeleton of the simplicial complex $\mc{S}_\varepsilon$.
Let
\[
d_0=B_1^\top\colon C^0(\Gamma;\R)\to C^1(\Gamma;\R),
\qquad
d_1=B_2^\top\colon C^1(\mc{S}_\varepsilon;\R)\to C^2(\mc{S}_\varepsilon;\R)
\]
be the first two coboundary maps.
With the standard Euclidean cochain inner products, the graph Laplacian on vertices and the $1$-Hodge Laplacian are
\[
L_0=d_0^*d_0=B_1B_1^\top,
\qquad
L_1=d_0d_0^*+d_1^*d_1=B_1^\top B_1+B_2B_2^\top.
\]
This is the standard combinatorial Hodge viewpoint on a simplicial complex \cite{Eckmann:1944}.
Assume that $\Gamma$ is connected, so that $0=\lambda_0<\lambda_1\leq\lambda_2\leq\dots$.
If $\psi_k\in C^0(\Gamma;\R)$ is a non-constant eigenvector of $L_0$ with eigenvalue $\lambda_k$, define
\[
\phi_k\coloneq d_0\psi_k=B_1^\top\psi_k\in C^1(\Gamma;\R).
\]
Then $\phi_k$ is an exact $1$-cochain and lies in the gradient part of the $1$-Hodge spectrum.
Indeed, $d_1\phi_k=d_1d_0\psi_k=0$, and
\[
L_1\phi_k
=d_0d_0^*d_0\psi_k+d_1^*d_1d_0\psi_k
=d_0L_0\psi_k
=\lambda_k\phi_k.
\]
On an oriented edge $e=(u,v)$, the cochain $\phi_k$ is the edge difference
\[
\phi_k(e)=\psi_k(v)-\psi_k(u).
\]
Therefore, for any oriented path $p$ from a base vertex $v_0$ to a vertex $v$, with signs chosen according to the path orientation,
\[
\int_p\phi_k\coloneq \sum_{e\in p}\phi_k(e)=\psi_k(v)-\psi_k(v_0).
\]
In this sense, the diffusion coordinate $\psi_k$ is obtained (up to an additive constant) by integrating the exact edge cochain $\phi_k$, representing an eigenvector of the Hodge Laplacian.

This is analogous to the construction of circular coordinates from persistent cohomology \cite{DeSilva2009persistent}, where one integrates a harmonic representative of an integer cohomology class.
For diffusion coordinates, the integrated $1$-cochain is exact, so its period around every cycle is zero and the resulting coordinate is a real-valued function on vertices.
For circular coordinates, the integrated $1$-cochain is harmonic and represents an integer cohomology class, so its periods around cycles may be non-zero integers and the resulting coordinate is well-defined only modulo $\Z$.
Both cases satisfy the local triangle condition $d_1\alpha=0$, so integration around the boundary of every $2$-simplex vanishes.
In contrast, coexact curl eigenvectors of the $1$-Hodge Laplacian generally do not lie in $\ker d_1$ and integrating them around triangle boundaries can give non-zero values, so their path integrals do not define node-level coordinates.

\section{Experiments}
\label{sec:experiments}

We first validate \textsc{circol} on synthetic data in \cref{sec:synthetic-experiment}.
We then test it on molecular dynamics data in \cref{sec:molecular-torsion}, where the task is to identify the torsion angles that explain the rotational symmetries in the molecular conformations.
Finally, we test it on neural data of head direction cells in \cref{sec:head-direction-experiment}.

\subsection{Synthetic point clouds}
\label{sec:synthetic-experiment}

\begin{figure}[t]
	\centering
	\includegraphics[width=\linewidth]{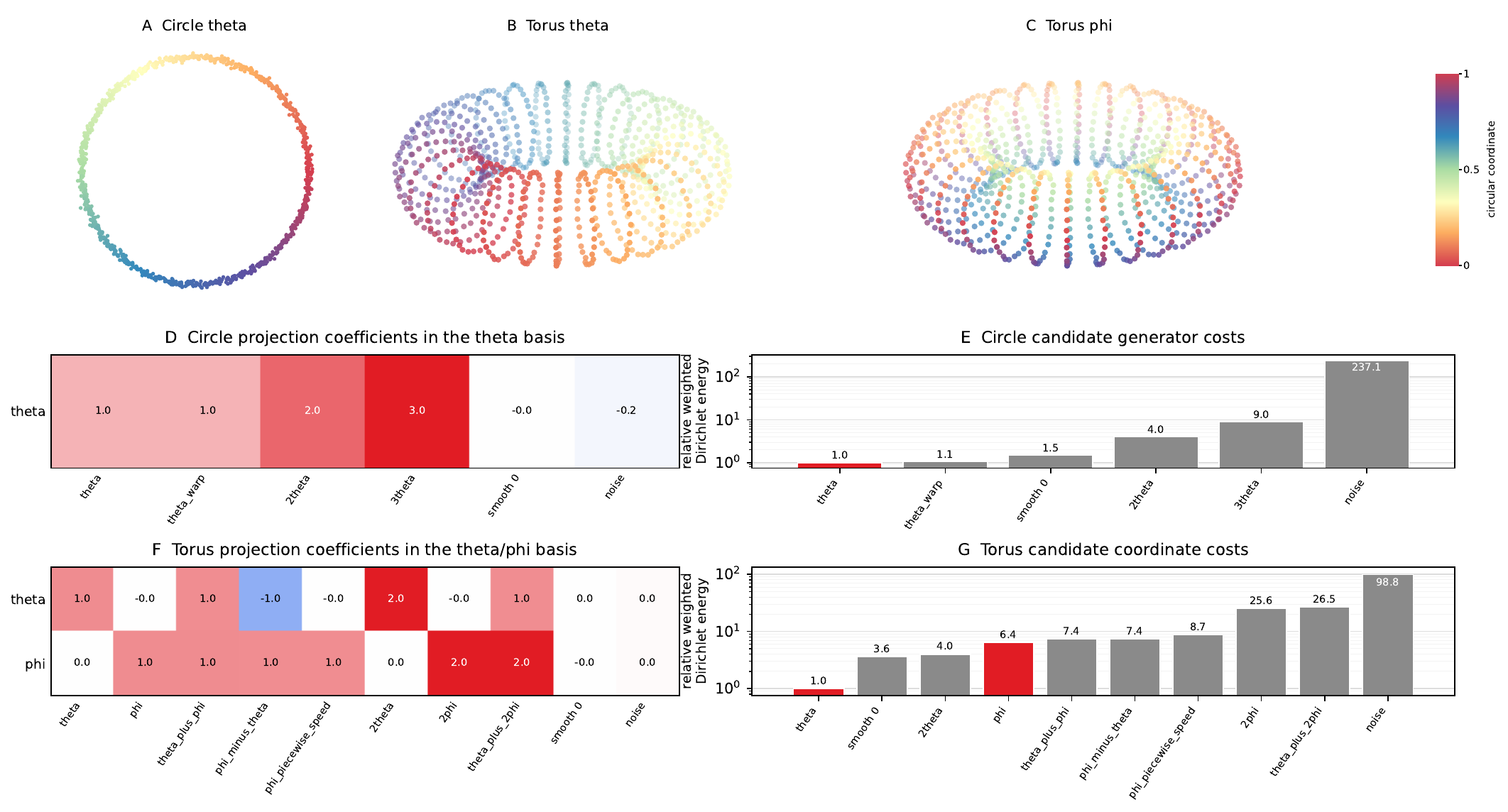}
	\caption[Synthetic validation]{\label{fig:synthetic-main}
		\textbf{Synthetic validation on a noisy circle (\emph{A}) and torus (\emph{B--C}).}
		Expected circular coordinates are shown in \emph{A--C}.
		Panel \emph{D} and \emph{F} show the projection matrices of the dictionary elements onto the harmonic representatives of the selected persistent classes, expressed in the basis of the ground truth cohomology classes.
		\emph{E} and \emph{G} show the energies of the dictionary elements relative to the selected ones (normalised to 1), with the selected elements marked in red.
	}
\end{figure}

We first test \textsc{circol} on controlled synthetic point clouds where the ground-truth circular coordinates are known.
The examples are a noisy circle with one cohomology generator and a torus with two.
The dictionaries contain the true coordinates, higher-winding alternatives such as $2\theta$ and $3\theta$, linear combinations in the case of the torus, non-linear transformations of the coordinates, smooth zero-degree coordinates, and noise.
In \cref{fig:synthetic-main} \emph{A--C}, we show the three ground truth circular coordinates on the circle and torus.

As described in \cref{alg:circular-dictionary-selection}, \textsc{circol} computes the persistent cohomology of the point cloud, returning one significant class for the circle and two significant classes for the torus.
Then for each feature, it constructs a simplicial complex at a scale where the feature is alive, computes inner product on cochains, and then projects the cohomology representatives to the discrete harmonic representatives of the selected classes.

For every dictionary element of circular coordinates, we then compute the associated $1$-cochain by taking angular differences along edges of the simplicial complex.
The inner product on cochains is used to project these dictionary cochains onto the space spanned by the harmonic representatives of the selected persistent classes, and we will then express these projections in the basis of the harmonic representatives.
We show the projection matrices in \cref{fig:synthetic-main} \emph{D} and \emph{F}, but use for the visualisation the identified ground truth cohomology basis.

Although not constrained to be integer-valued, the projections recover the expected integer winding numbers: on the circle, $\theta$ and the warped $\theta$, which is theta with non-uniform winding speed, have degree one, $2\theta$ has degree two, $3\theta$ has degree three, and zero-degree functions project to zero.
Only the noise coordinate has a significant non-integer projection, but it is still significantly below $1$.
This is to be expected, as our theorem for integer projections and winding numbers required continuous circular coordinates and thus well-defined and curl-free cochain lifts.

On the torus, the same computation recovers the expected integer degree vectors for $\theta$, $\phi$, and mixed coordinates such as $\theta+\phi$.
Thus the finite-sample inner product and projection step behave as predicted by the continuous formulation.

The selection step then chooses the lowest-energy full-rank basis.
On the circle, $2\theta$ and $3\theta$ span the same real cohomology but have approximately four and nine times the energy of $\theta$, as expected from their winding numbers.
The warped $\theta$ has slightly higher energy than $\theta$, and the noise coordinate has very high energy/cost, so \textsc{circol} selects $\theta$ as the best explanation of the persistent loop.
On the torus, the algorithm selects $\theta$ and $\phi$ rather than more expensive mixed or higher-winding alternatives.

\subsection{Molecular conformations from quantum molecular dynamics}

\label{sec:molecular-torsion}

\begin{figure}[t]
	\centering
	\includegraphics[width=\linewidth]{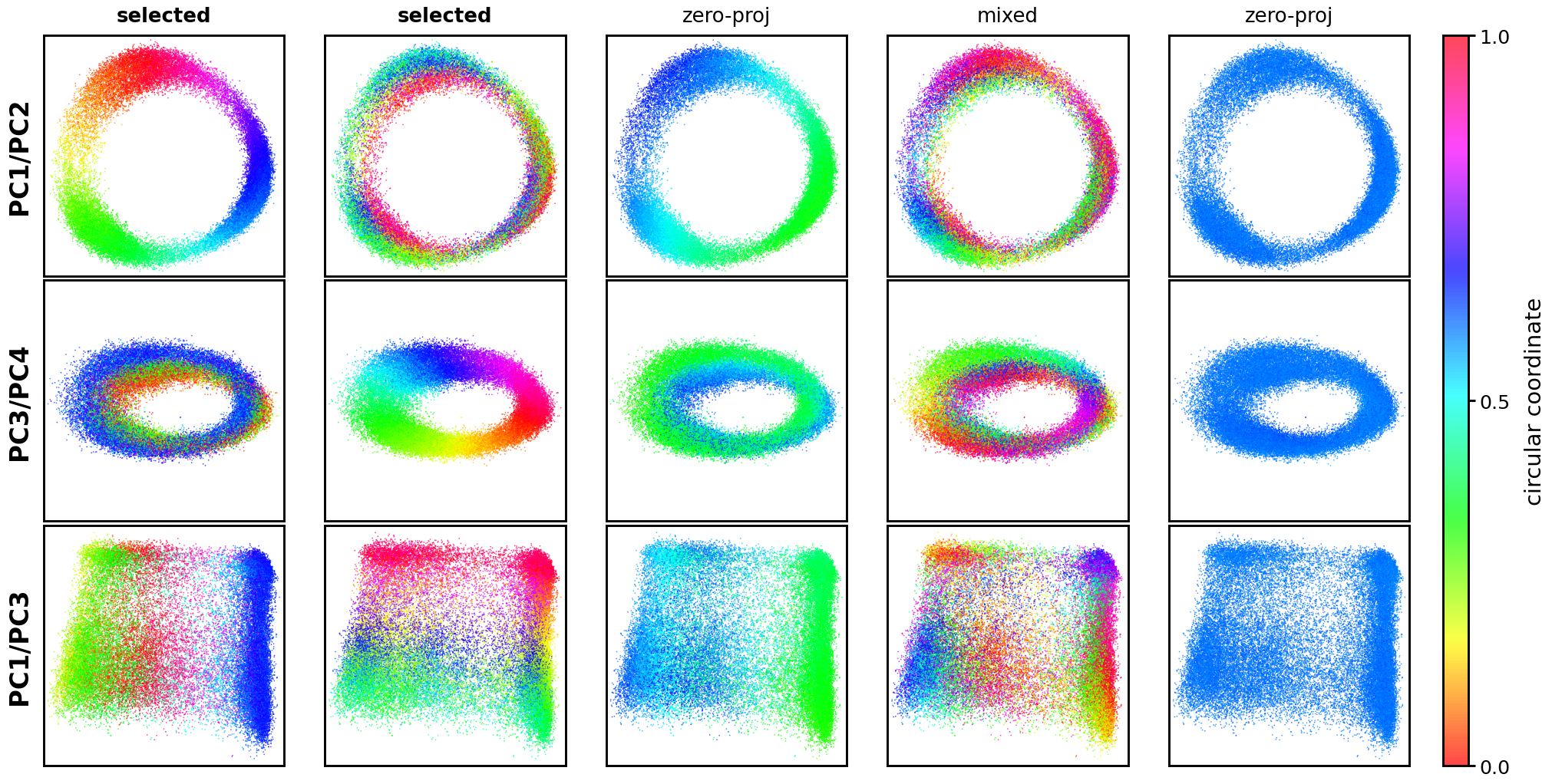}
	\caption[Ethanol torsion candidates]{\label{fig:ethanol-candidates}
		\textbf{Ethanol configurations coloured by representative torsion candidates.}
		The $y$-axis distinguishes different projections on principal components of the data, with PCA 1--2 and PCA 3--4 showing the two topological features.
		\emph{Selected:} The two selected torsions, which are the lowest-energy explanations of the two persistent classes, show an explanation of the PC1--PC2 loop and the PC3--PC4 loop respectively.
		\emph{zero-proj:} Two torsion candidates with trivial projection to harmonic generators. As can be seen, these coordinates do not wrap around the loops in the data and thus do not explain the persistent classes.
		\emph{mixed:} Torsion candidate that is a superposition of both base cohomology classes. It has non-trivial projection to both harmonic generators, but is not selected because it has higher energy than the two base torsions, essentially amounting to a \enquote{correct} but \enquote{overly complicated} explanation of topology.
	}
\end{figure}
\begin{figure}[t]
	\centering
	\includegraphics[width=0.58\linewidth]{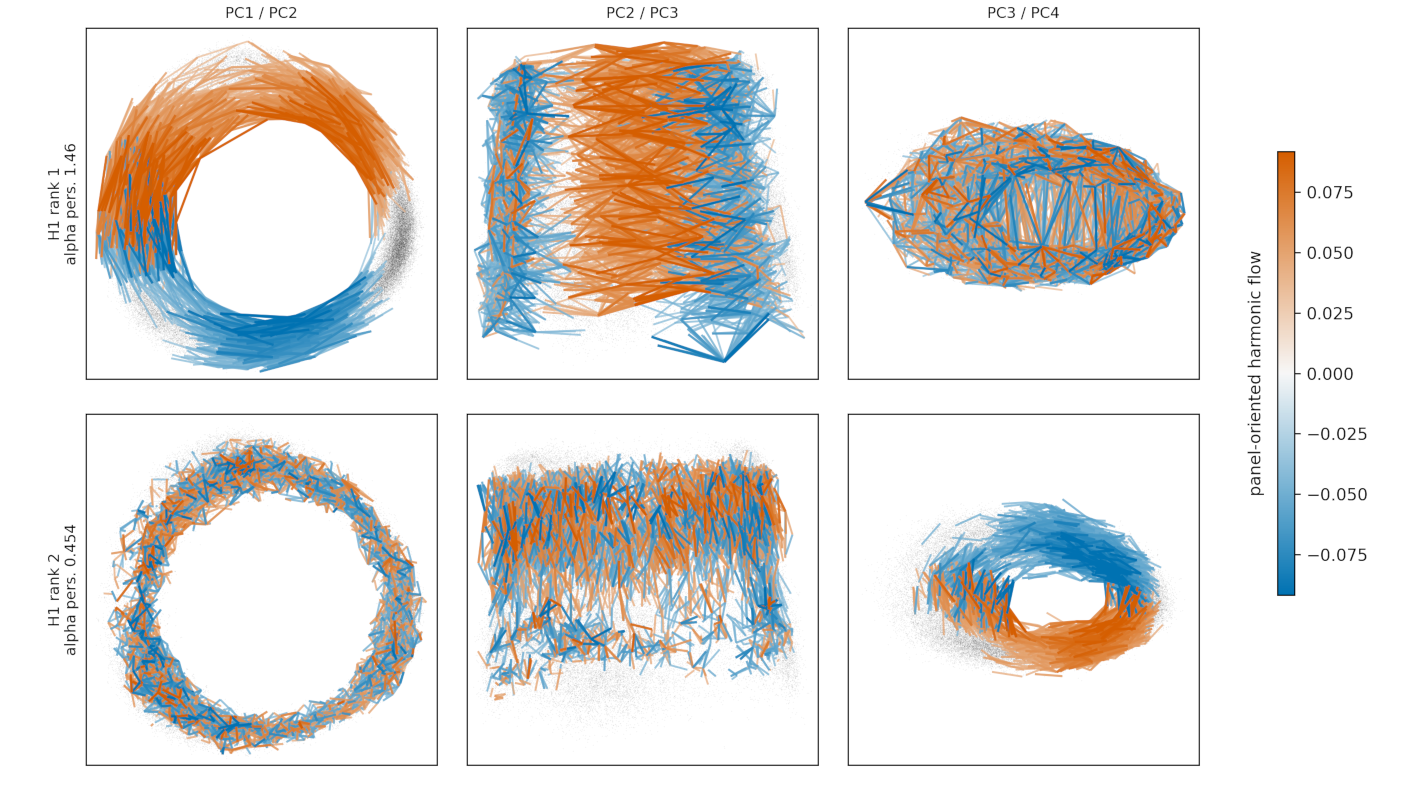}
	\includegraphics[width=0.4\linewidth]{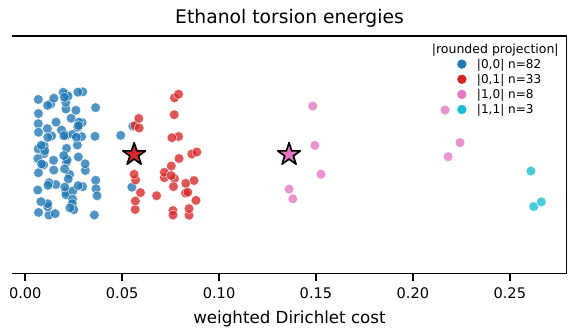}
	\caption[Ethanol harmonic representatives and energies]{\label{fig:costs_harmonic_ethanol}
	\textbf{Ethanol harmonic representatives and energies}
		\emph{Left:} The harmonic representative of the two persistent classes for ethanol. Colour indicates orientation of the cochain on edges of the simplicial complex, only edges with strongest signal are drawn.
		\emph{Right:} The Dirichlet energy of all candidate torsions, with the selected torsion marked in red.
		Torsion candidates are grouped according to their (absolute value of) coordinates in the projection onto the harmonic basis of the selected persistent cohomology classes.
		While $(0,0)$ projections have low Dirichlet energy, they do not contribute to spanning a basis of the desired harmonic subspace.
		Circular coordinates representing a superposition of both persistent cohomology classes have the largest Dirichlet energy.
	}
\end{figure}
We next test \textsc{circol} on molecular dynamics data from Chmiela et al.\ \cite{chmiela2018exact}, using the ethanol, toluene, and malonaldehyde data.
Each observation is a molecular conformation represented as a high-dimensional vector, together with energy and force labels in the original dataset.
For our experiment, however, the selection step only sees the coordinates of the sampled point cloud and a large dictionary of candidate circular functions, obtained by evaluating many possible torsion angles on the same conformations.
We show PCA projections of the dataset and a selection of candidate torsions in \cref{fig:ethanol-candidates}.

Functional-group rotations and coupled dihedral motions create circular or toroidal structure in the conformational ensemble, but this structure is embedded nonlinearly in a high-dimensional cloud and can be obscured by other degrees of freedom.
Moreover, a molecule can have many plausible torsion candidates, most of which are chemically valid angles but do not explain the dominant topology of the sampled trajectory.
The task is to decide which candidate torsion, among a long list of alternatives, accounts for that topology.

These data were also analysed in ManifoldLasso and TSLasso \cite{koelle2022manifold,koelle2024consistency}.
However, the authors in the cited works had to treat the torsion angles as arbitrary real-valued coordinate functions, leading to a non-continuity where the angle wraps around at $2\pi$ and could not measure nor guarantee the correct cohomological behaviour of the torsion candidates.

For ethanol, toluene, and malonaldehyde, \textsc{circol} first extracts the persistent circular feature from the point cloud and then projects all candidate torsion cochains onto the corresponding harmonic representative.
The selected coordinate is the lowest-energy dictionary element whose projection explains the persistent class.
For ethanol, \cref{fig:costs_harmonic_ethanol} shows the two harmonic representatives used for these projections on the left and the corresponding energy comparison of the torsion candidates on the right.
In this way, \textsc{circol} recovers the physically meaningful torsional symmetry from topology alone, using the candidate torsions only as possible interpretations and without using any ground-truth torsion label during selection.

\subsection{Head-direction neurons}
\label{sec:head-direction-experiment}

\begin{figure}[t]
	\centering
	\includegraphics[width=0.8\linewidth]{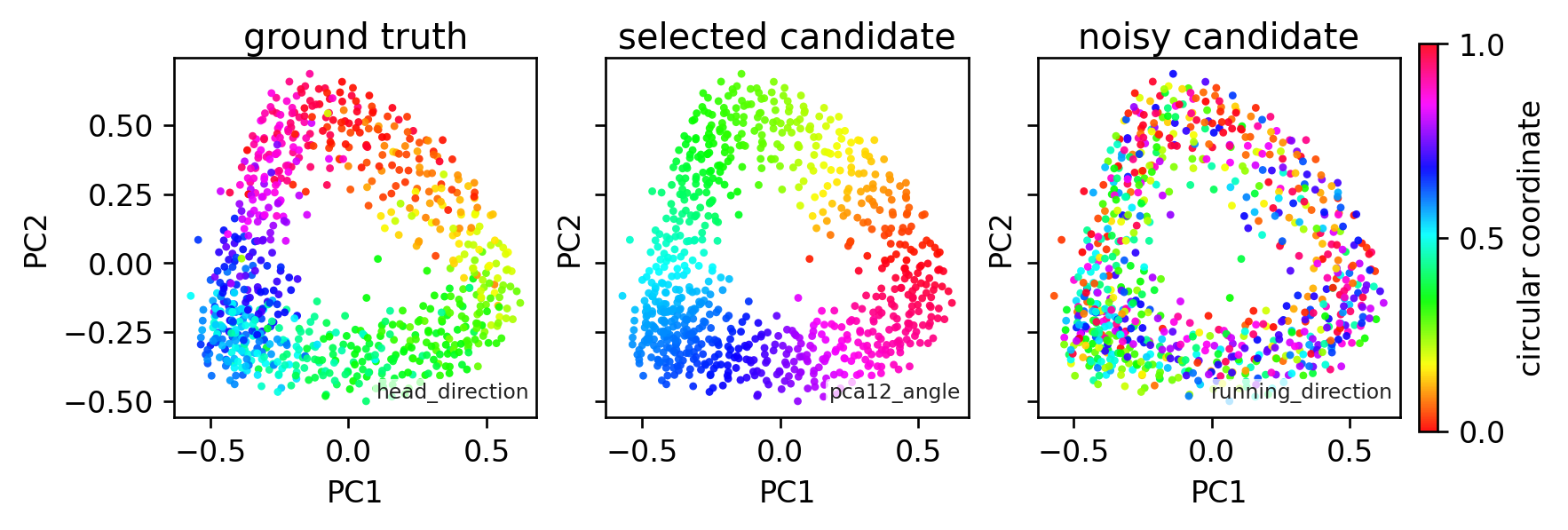}
	\includegraphics[width=0.8\linewidth]{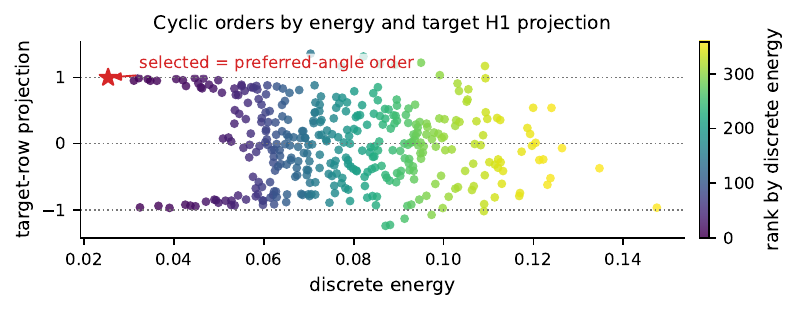}
	\caption[Head-direction experiment]{\label{fig:head-direction-coordinates}
	\label{fig:head-direction-cyclic-rank}
	\textbf{Head-direction experiment:} \emph{Top:}	
		The selected cyclic-order decoder varies smoothly in the first two PCA coordinates and roughly agrees with the tracked head direction up to orientation.
		A random cyclic order of the same seven units is visibly noisy and does not explain the persistent loop.
		\emph{Bottom:}
		Cohomology projection and energy of all $360$ cyclic-order candidates of the seven selected neurons.
		The red marker is both the selected candidate and the ground-truth ordering.
		The $x$-axis is the Dirichlet energy of the cochain associated to the candidate, and thus measures the cost in the optimisation problem.
		The $y$-axis is the projection of the candidate cochain onto the harmonic representative of the selected persistent class, and thus measures how well the candidate explains the persistent loop.
		While for perfect continuous circular coordinates, this projection is expected to be an integer, for increasingly noisy candidate cochains the projection degrades from pure $-1/0/1$ values.
	}
\end{figure}

Many animals maintain an internal sense of orientation while they move.
In the postsubiculum, head-direction cells encode this orientation by firing at preferred angles, making this population a natural biological encoding of a circular coordinate \cite{taube1990headI}.

We use a curated dataset of mouse postsubiculum recordings \cite{duszkiewicz2025dandiHeadDirection}, derived from \cite{duszkiewicz2024local,duszkiewicz2025figshareLocalOrigin}.
We analyse one open-field epoch, which contains spike times, curated head-direction labels, the animal's position, and the tracked head direction.
The point cloud is built from population activity of annotated head-direction cells after standard binning, smoothing, normalisation, and PCA.
Thus the ground-truth circular variable is known, but the geometry on which \textsc{circol} operates comes from nonlinear neural activity rather than from the behavioural angle itself.

We will first do a dictionary selection with the tracked head direction included in the dictionary to confirm that \textsc{circol} identifies the expected circular coordinate.
It assigns a coefficient close to $1$ to this coordinate on the selected persistent class, confirming that the loop found in the neural point cloud is the expected orientation circle.
In a setting where the meaning of the loop was not known, this would identify which measured circular quantity explains the topology; in our setting, this is another proof of concept that the method works as expected.
The agreement between the neural coordinates and the behavioural head direction is visible in \cref{fig:head-direction-coordinates}.

We then removed the tracked angle from the candidate dictionary:
We choose seven head-direction cells whose preferred angles cover the circle, but then pretend that their cyclic order and exact firing patterns are unknown.
We build one circular candidate for each possible cyclic ordering of these seven cells, up to rotations and reversals.
Then, we test \textsc{circol} to identify the ground-truth cyclic order of the seven cells from the neural point cloud alone, without any behavioural or head-direction reference or ground-truth annotations.

Among the $360=7!/(7\cdot 2)$ resulting candidates, \textsc{circol} selects exactly the correct permutation.
\Cref{fig:head-direction-cyclic-rank} shows all cyclic-order candidates, their projection to the cohomology class,  and marks the selected ordering.
Even when scaling up to hundreds, or even thousands or more, of candidate cyclic orders, \textsc{circol} is linear in the number of candidates, as the persistent cohomology and harmonic representatives are computed only once, and the projection and selection steps are linear in the number of candidates.

Thus, \textsc{circol} works on real-world neural data and correctly infers neural ordering simply from point cloud topology.
\FloatBarrier

\section{Discussion}
\label{sec:discussion}
\paragraph{Summary}
We have made the following contributions:
\begin{enumerate}
	\item We have presented \textsc{circol}, a method for dictionary selection for circular coordinates using persistent cohomology, comparing reference (persistent) cohomology classes to candidate circular coordinates by the classes of their pulled-back angular forms.
	In particular, we translate minimal-energy selection in this setting as a minimum-weight basis problem in a vector matroid.
	\item \textsc{circol} derives a finite-sample approximation using persistent cohomology, discrete harmonic representatives, and a density-corrected cochain inner product.
	\item We show that the inner product on simplicial cochains converges to the continuous inner product on differential forms on the underlying manifold in the infinite-sample limit.
	\item We use the obtained metrics both as a selector and a diagnostic for unexplained persistent classes, topologically trivial candidates, and the relation between candidate circular coordinates.
	\item We validate \textsc{circol} on synthetic examples, on molecular simulations and on head-direction neural data with circular structure.
\end{enumerate}
\paragraph{Limitations}
\textsc{circol} starts from persistent cohomology, and therefore depends on the circular structure being captured by a strong enough $H^1$ signal at the relevant scale.
This also means that \textsc{circol} incurs the computational cost of persistent cohomology, which does not scale to arbitrarily large point clouds.
Furthermore, many meaningful features in data are not topological circles or measurable by $H^1$ cohomology.
While expected, this means that our method is not applicable in many real-world datasets without topological signal.
Furthermore, our consistency result, which is similar in form to standard consistency results for other operators, is asymptotic in a regime where \emph{both} $h_N\to 0$ \emph{and} $N h_N^d/\log N\to\infty$.
This means that we expect not only our total sample size, but also the number of neighbors per point, to go to infinity.
In computations, however, one usually tries to keep the neighbourhood size modest even for large sample sizes to preserve sparsity and computational efficiency.
\paragraph{Future work}
Estimators for the cochain inner product with guarantees in low-sample settings are therefore an important direction.
A further direction is to relax the clean manifold model, since experimental data may be noisy, stratified, or only approximately manifold-like.
In line with usual consistency results in the literature, we assume that the data are sampled from a manifold without noise.
Finally, robust variants should allow almost circular coordinates coming from almost harmonic representatives, so that high-noise cyclic structure can be compared with the dictionary without requiring exact circular cohomology.
This could, for example, catch topological structures where the hole is \enquote{obfuscated} by a small number of outliers.


\section*{Acknowledgements}
First of all, \textsc{vpg} thanks his advisor Michael T.\ Schaub for frequent discussions and a generous support on very many levels.
\textsc{vpg} acknowledges support of the \textsc{dfg} (\enquote{Deutsche Forschungsgemeinschaft}) within Research Training Group 2236, \textsc{unravel}, and support of the European Union (\textsc{erc}, \textsc{high-hopes}, 101039827).
Views and opinions expressed are however those of the author(s) only and do not necessarily reflect those of the European Union or the European Research Council Executive Agency. Neither the European Union nor the granting authority can be held responsible for them.
Part of this project was conducted while on a research visit supported by a scholarship of the \textsc{daad} (German Academic Exchange Service).
This research was undertaken, in part, thanks to support from the Canada Research Chairs Program and from the Vector Institute for AI, through the \textsc{cifar} Pan-Canadian AI Strategy, the Government of Ontario, and leading industry sponsors from across the Canadian economy.

\printbibliography

@article{paik2023circular,
  title   = {Circular coordinates for density-robust analysis},
  author  = {Paik, Taejin and Park, Jaemin},
  journal = {arXiv preprint arXiv:2301.12742},
  year    = {2023}
}

@article{taube1990headI,
  title   = {Head-direction cells recorded from the postsubiculum in freely moving rats. I. Description and quantitative analysis},
  author  = {Taube, Jeffrey S. and Muller, Robert U. and Ranck, James B., Jr.},
  journal = {The Journal of Neuroscience},
  volume  = {10},
  number  = {2},
  pages   = {420--435},
  year    = {1990},
  doi     = {10.1523/JNEUROSCI.10-02-00420.1990}
}

@article{duszkiewicz2024local,
  title   = {Local origin of excitatory-inhibitory tuning equivalence in a cortical network},
  author  = {Duszkiewicz, Adrian J. and Orhan, Pierre and Skromne Carrasco, Sofia and Brown, Eleanor H. and Owczarek, Eliott and Vite, Gilberto R. and Wood, Emma R. and Peyrache, Adrien},
  journal = {Nature Neuroscience},
  volume  = {27},
  pages   = {782--792},
  year    = {2024},
  doi     = {10.1038/s41593-024-01588-5}
}

@misc{duszkiewicz2025dandiHeadDirection,
  title     = {Large-scale recordings of head direction cells in mouse postsubiculum},
  author    = {Duszkiewicz, Adrian and Skromne Carrasco, Sofia and Peyrache, Adrien},
  year      = {2025},
  publisher = {DANDI Archive},
  version   = {0.250207.0025},
  doi       = {10.48324/dandi.000939/0.250207.0025},
  url       = {https://dandiarchive.org/dandiset/000939/0.250207.0025},
  note      = {Data set}
}

@misc{duszkiewicz2025figshareLocalOrigin,
  title     = {Local origin of excitatory-inhibitory tuning equivalence in a cortical network (Duszkiewicz et al. 2024)},
  author    = {Duszkiewicz, Adrian},
  year      = {2025},
  publisher = {figshare},
  doi       = {10.6084/m9.figshare.24921252},
  url       = {https://doi.org/10.6084/m9.figshare.24921252},
  note      = {Data set and code}
}

@article{vandereyken2023methods,
  title     = {Methods and applications for single-cell and spatial multi-omics},
  author    = {Vandereyken, Katy and Sifrim, Alejandro and Thienpont, Bernard and Voet, Thierry},
  journal   = {Nature Reviews Genetics},
  volume    = {24},
  number    = {8},
  pages     = {494--515},
  year      = {2023},
  publisher = {Nature Publishing Group UK London}
}

@article{cheng2024phlower,
  author  = {Cheng, Mingbo and Jansen, Jitske and Reimer, Katharina and Grande, Vincent P. and Nagai, James Shiniti and Li, Zhijian and Kie{\ss}ling, Paul and Grasshoff, Martin and Kuppe, Christoph and Schaub, Michael T. and Kramann, Rafael and Costa, Ivan G.},
  title   = {{PHLOWER} leverages single-cell multimodal data to infer complex, multi-branching cell differentiation trajectories},
  year    = {2025},
  journal = {Nature Methods}
}

@article{Maggs2025TopologyConcurrentCyclicProcesses,
  author  = {Maggs, Kelly and Youssef, Markus K. and Pulver, Cyril and Isma, Jovan and Nguy{\^e}n, T{\^a}m J. and Arzt, Matthias and Karthaus, Wouter and Harrington, Heather A. and Hess, Kathryn and Dotto, G. Paolo},
  title   = {Topology identifies concurrent cyclic processes in single-cell transcriptomics and androgen receptor function},
  journal = {bioRxiv},
  year    = {2025},
  date    = {2025-12-08},
  doi     = {10.1101/2025.01.09.632214},
  url     = {https://doi.org/10.1101/2025.01.09.632214},
  note    = {Preprint, version 3}
}

@article{maehara2019modeling,
  title     = {Modeling latent flows on single-cell data using the Hodge decomposition},
  author    = {Maehara, Kazumitsu and Ohkawa, Yasuyuki},
  journal   = {bioRxiv},
  pages     = {592089},
  year      = {2019},
  publisher = {Cold Spring Harbor Laboratory}
}

@article{su2024hodge,
  title     = {Hodge decomposition of single-cell RNA velocity},
  author    = {Su, Zhe and Tong, Yiying and Wei, Guo-Wei},
  journal   = {Journal of chemical information and modeling},
  volume    = {64},
  number    = {8},
  pages     = {3558--3568},
  year      = {2024},
  publisher = {ACS Publications}
}

@article{jones2024manifold,
  title   = {Manifold Diffusion Geometry: Curvature, Tangent Spaces, and Dimension},
  author  = {Jones, Iolo},
  journal = {arXiv preprint arXiv:2411.04100},
  year    = {2024}
}

@article{jones2024diffusion,
  title   = {Diffusion Geometry},
  author  = {Jones, Iolo},
  journal = {arXiv preprint arXiv:2405.10858},
  year    = {2024},
  doi     = {10.48550/arXiv.2405.10858}
}

@inproceedings{grande2025point,
  title     = {Point-Level Topological Representation Learning on Point Clouds},
  author    = {Grande, Vincent Peter and Schaub, Michael T},
  year      = {2025},
  booktitle = {Forty-second International Conference on Machine Learning}
}

@article{grandeNotAllSmallEigenvalues,
  title         = {Disentangling the Spectral Properties of the Hodge Laplacian: Not All Small Eigenvalues Are Equal},
  author        = {Grande, Vincent P. and Schaub, Michael T.},
  journal       = {arXiv preprint arXiv:2311.14427},
  year          = {2024},
  eprint        = {2311.14427},
  archiveprefix = {arXiv},
  doi           = {10.48550/arXiv.2311.14427}
}

@article{berry2020spectral,
  title     = {Spectral exterior calculus},
  author    = {Berry, Tyrus and Giannakis, Dimitrios},
  journal   = {Communications on Pure and Applied Mathematics},
  volume    = {73},
  number    = {4},
  pages     = {689--770},
  year      = {2020},
  publisher = {Wiley Online Library}
}

@article{coifman2006diffusion,
  title     = {Diffusion maps},
  author    = {Coifman, Ronald R and Lafon, St{\'e}phane},
  journal   = {Applied and computational harmonic analysis},
  volume    = {21},
  number    = {1},
  pages     = {5--30},
  year      = {2006},
  publisher = {Elsevier}
}

@article{belkin2003laplacian,
  title     = {Laplacian eigenmaps for dimensionality reduction and data representation},
  author    = {Belkin, Mikhail and Niyogi, Partha},
  journal   = {Neural computation},
  volume    = {15},
  number    = {6},
  pages     = {1373--1396},
  year      = {2003},
  publisher = {MIT Press}
}

@inproceedings{koelle2024consistency,
  title        = {Consistency of dictionary-based manifold learning},
  author       = {Koelle, Samson J and Zhang, Hanyu and Murad, Octavian-Vlad and Meila, Marina},
  booktitle    = {International Conference on Artificial Intelligence and Statistics},
  pages        = {4348--4356},
  year         = {2024},
  organization = {PMLR}
}

@article{koelle2022manifold,
  title   = {Manifold coordinates with physical meaning},
  author  = {Koelle, Samson J and Zhang, Hanyu and Meila, Marina and Chen, Yu-Chia},
  journal = {Journal of Machine Learning Research},
  volume  = {23},
  number  = {133},
  pages   = {1--57},
  year    = {2022}
}

@inproceedings{scoccola2023toroidal,
  author    = {Scoccola, Luis and Gakhar, Hitesh and Bush, Johnathan and Schonsheck, Nikolas and Rask, Tatum and Zhou, Ling and Perea, Jose A.},
  title     = {Toroidal Coordinates: Decorrelating Circular Coordinates with Lattice Reduction},
  booktitle = {39th International Symposium on Computational Geometry (SoCG 2023)},
  editor    = {Chambers, Erin W. and Gudmundsson, Joachim},
  series    = {Leibniz International Proceedings in Informatics (LIPIcs)},
  isbn      = {978-3-95977-273-0},
  issn      = {1868-8969},
  volume    = {258},
  pages     = {57:1--57:20},
  year      = {2023},
  publisher = {Schloss Dagstuhl -- Leibniz-Zentrum f{\"u}r Informatik},
  address   = {Dagstuhl, Germany},
  doi       = {10.4230/LIPIcs.SoCG.2023.57},
  url       = {https://drops.dagstuhl.de/entities/document/10.4230/LIPIcs.SoCG.2023.57}
}

@article{schonsheck2024spherical,
  title     = {Spherical coordinates from persistent cohomology},
  author    = {Schonsheck, Nikolas C and Schonsheck, Stefan C},
  journal   = {Journal of Applied and Computational Topology},
  volume    = {8},
  number    = {1},
  pages     = {149--173},
  year      = {2024},
  publisher = {Springer}
}

@article{helein2008harmonic,
  title     = {Harmonic maps},
  author    = {H{\'e}lein, Fr{\'e}d{\'e}ric and Wood, John C},
  journal   = {Handbook of global analysis},
  volume    = {1213},
  pages     = {417--491},
  year      = {2008},
  publisher = {Elsevier Amsterdam}
}

@book{baird2003harmonic,
  title     = {Harmonic morphisms between Riemannian manifolds},
  author    = {Baird, Paul and Wood, John C},
  number    = {29},
  year      = {2003},
  publisher = {Oxford University Press}
}

@article{DeSilva2009persistent,
  author  = {de Silva, Vin and Morozov, Dmitriy and Vejdemo-Johansson, Mikael},
  title   = {Persistent Cohomology and Circular Coordinates},
  journal = {Discrete \& Computational Geometry},
  volume  = {45},
  number  = {4},
  pages   = {737--759},
  year    = {2011},
  doi     = {10.1007/s00454-011-9344-x}
}

@incollection{Perea2020,
  author    = {Perea, Jose A.},
  editor    = {Baas, Nils A.
               and Carlsson, Gunnar E.
               and Quick, Gereon
               and Szymik, Markus
               and Thaule, Marius},
  title     = {Sparse Circular Coordinates via Principal {$\Z$}-Bundles},
  booktitle = {Topological Data Analysis},
  series    = {Abel Symposia},
  volume    = {15},
  year      = {2020},
  publisher = {Springer International Publishing},
  address   = {Cham},
  pages     = {435--458},
  doi       = {10.1007/978-3-030-43408-3_17},
  url       = {https://doi.org/10.1007/978-3-030-43408-3_17}
}

@article{perea2017multiscaleProjective,
  author  = {Perea, Jose A.},
  title   = {Multiscale Projective Coordinates via Persistent Cohomology of Sparse Filtrations},
  journal = {Discrete \& Computational Geometry},
  volume  = {59},
  number  = {1},
  pages   = {175--225},
  year    = {2018},
  doi     = {10.1007/s00454-017-9927-2}
}

@article{rybakken2019decoding,
  author  = {Rybakken, Erik and Baas, Nils A. and Dunn, Benjamin A.},
  title   = {Decoding of Neural Data Using Cohomological Feature Extraction},
  journal = {Neural Computation},
  volume  = {31},
  number  = {1},
  pages   = {68--93},
  year    = {2019}
}

@article{gardner2022toroidal,
  author  = {Gardner, Richard J. and Hermansen, Erik and Pachitariu, Marius and Burak, Yoram and Baas, Nils A. and Dunn, Benjamin A. and Moser, May-Britt and Moser, Edvard I.},
  title   = {Toroidal topology of population activity in grid cells},
  journal = {Nature},
  volume  = {602},
  number  = {7895},
  pages   = {123--128},
  year    = {2022},
  doi     = {10.1038/s41586-021-04268-7}
}

@article{perea2015sliding,
  author  = {Perea, Jose A. and Harer, John},
  title   = {Sliding Windows and Persistence: An Application of Topological Methods to Signal Analysis},
  journal = {Foundations of Computational Mathematics},
  volume  = {15},
  number  = {3},
  pages   = {799--838},
  year    = {2015},
  doi     = {10.1007/s10208-014-9206-z}
}

@article{Eckmann:1944,
  author  = {Eckmann, Beno},
  journal = {Commentarii mathematici Helvetici},
  pages   = {240--255},
  title   = {{Harmonische Funktionen und Randwertaufgaben in einem Komplex.}},
  volume  = {17},
  date    = {1944/1945}
}

@article{Kovacev-Nikolic:2016,
  title   = {Using persistent homology and dynamical distances to analyze protein binding},
  volume  = {15},
  doi     = {10.1515/sagmb-2015-0057},
  number  = {1},
  journal = {Statistical Applications in Genetics and Molecular Biology},
  author  = {Kovacev-Nikolic, Violeta and Bubenik, Peter and Nikoli{\'c}, Dragan and Heo, Giseon},
  year    = {2016},
  month   = {1}
}

@article{edmonds1971matroids,
  title     = {Matroids and the greedy algorithm},
  author    = {Edmonds, Jack},
  journal   = {Mathematical programming},
  volume    = {1},
  number    = {1},
  pages     = {127--136},
  year      = {1971},
  publisher = {Springer}
}

@article{chmiela2018exact,
  title   = {Towards exact molecular dynamics simulations with machine-learned force fields},
  author  = {Chmiela, Stefan and Sauceda, Huziel E. and M{\"u}ller, Klaus-Robert and Tkatchenko, Alexandre},
  journal = {Nature Communications},
  volume  = {9},
  number  = {1},
  pages   = {3887},
  year    = {2018},
  doi     = {10.1038/s41467-018-06169-2}
}
\appendix
\section{Proof of Convergence of the $1$-cochain inner product}
\label{sec:circular-convergence}
In this section, we will give a proof of the convergence theorem for the density-corrected cochain inner product, \cref{thm:cochain-inner-product-consistency-main-text}.

Let $d=\dim \mc{M}$, let $\mu$ be the Riemannian volume measure on $\mc{M}$, and let $\alpha,\beta\in \Omega^1(\mc{M})$ be smooth. Our target quantity is
\begin{equation}
\langle \alpha,\beta\rangle_{L_2(\mu)}\coloneq \int_{\mc{M}} \langle \alpha_x,\beta_x\rangle_g\,d\mu(x),
\end{equation}
where $\langle \cdot,\cdot\rangle_g$ is the pointwise inner product on $T_x^*\mc{M}$ induced by the Riemannian metric.

To connect this to data, we use only information carried by nearby pairs of points. Write $\mathrm{inj}(\mc{M})$ for the injectivity radius of $\mc{M}$. If $x,y\in \mc{M}$ satisfy $d_{\mc{M}}(x,y)<\mathrm{inj}(\mc{M})$, there is a unique minimizing unit-speed geodesic
\begin{equation}
\gamma_{x,y}\colon [0,d_{\mc{M}}(x,y)]\to \mc{M}
\end{equation}
from $x$ to $y$. We define
\begin{equation}
\alpha(x,y)\coloneq \int_0^{d_{\mc{M}}(x,y)} \alpha_{\gamma_{x,y}(t)}\!\bigl(\dot{\gamma}_{x,y}(t)\bigr)\,dt,
\qquad
\beta(x,y)\coloneq \int_0^{d_{\mc{M}}(x,y)} \beta_{\gamma_{x,y}(t)}\!\bigl(\dot{\gamma}_{x,y}(t)\bigr)\,dt.
\end{equation}
These are the continuum quantities that the discrete $1$-cochains will later approximate on short edges. Reversing the geodesic shows that
\begin{equation}
\alpha(y,x)=-\alpha(x,y),
\qquad
\beta(y,x)=-\beta(x,y).
\end{equation}

\paragraph{Step 1: Recover the pointwise inner product from directional averages}
At a fixed point $x$, the quantity $\alpha_x(\theta)$ is the evaluation of the covector $\alpha_x$ on the unit direction $\theta$. Averaging products of these directional evaluations over all unit directions recovers the full inner product.

\begin{lemma}[Sphere-average identity]
Let $x\in \mc{M}$ and let $\alpha_x,\beta_x\in T_x^*\mc{M}$.
Let $S_x\mc{M}\subset T_x\mc{M}$ be the unit sphere, equipped with surface measure $d\sigma_x$, and let
\begin{equation}
\omega_{d-1}\coloneq \int_{S_x\mc{M}} 1\,d\sigma_x
\end{equation}
be the volume of the Euclidean unit sphere in dimension $d$. Then
\begin{equation}
\int_{S_x\mc{M}} \alpha_x(\theta)\beta_x(\theta)\,d\sigma_x(\theta)
=
\frac{\omega_{d-1}}{d}\langle \alpha_x,\beta_x\rangle_g
\end{equation}
\end{lemma}
\begin{proof}
Choose an orthonormal basis of $T_x\mc{M}$ and identify $T_x\mc{M}$ with $\R^d$.
Write $\alpha_x(\theta)=\sum_i a_i\theta_i$ and $\beta_x(\theta)=\sum_j b_j\theta_j$.
Then
\begin{equation}
\int_{S_x\mc{M}} \alpha_x(\theta)\beta_x(\theta)\,d\sigma_x(\theta)
=
\sum_{i,j} a_i b_j \int_{S^{d-1}} \theta_i\theta_j\,d\sigma(\theta).
\end{equation}
By symmetry, the mixed terms vanish and all diagonal terms are equal. Since $\sum_{i=1}^d \theta_i^2=1$ on $S^{d-1}$, each diagonal integral equals $\omega_{d-1}/d$, and the claim follows.
\end{proof}

\paragraph{Step 2: Replace infinitesimal directions by short geodesic integrals}
The previous lemma describes what happens at a single tangent space. To turn it into a pairwise formula, we average over a small geodesic neighbourhood.

Let $\kappa\colon [0,\infty)\to [0,\infty)$ be a bounded measurable kernel profile supported in $[0,1)$ and not identically zero, and define
\begin{equation}
m_0\coloneq \omega_{d-1}\int_0^\infty \kappa(\rho)\rho^{d-1}\,d\rho,
\qquad
m_2\coloneq \omega_{d-1}\int_0^\infty \kappa(\rho)\rho^{d+1}\,d\rho.
\end{equation}
For $h<\mathrm{inj}(\mc{M})$, define the bandwidth-$h$ kernel
\begin{equation}
K_h(x,y)\coloneq h^{-d}\kappa\!\left(\frac{d_{\mc{M}}(x,y)}{h}\right).
\end{equation}
Because $\kappa$ is supported in $[0,1)$, the condition $K_h(x,y)\neq 0$ forces $d_{\mc{M}}(x,y)<h$, so the geodesic integrals above are well-defined on the support of the kernel.

\begin{lemma}[Short-range kernel localization]
Let $(\mc{M},g)$ be a smooth compact $d$-dimensional Riemannian manifold without boundary, let $\alpha,\beta\in \Omega^1(\mc{M})$ be smooth, let $\kappa\colon [0,\infty)\to [0,\infty)$ be bounded and supported in $[0,1)$, and let $0<h<\mathrm{inj}(\mc{M})$. Then, for every $x\in \mc{M}$,
\begin{equation}
\int_{\mc{M}} K_h(x,y)\alpha(x,y)\beta(x,y)\,d\mu(y)
=
h^2\frac{m_2}{d}\langle \alpha_x,\beta_x\rangle_g + O(h^4)
\end{equation}
as $h\to 0$, uniformly in $x$.
\end{lemma}
\begin{proof}
Fix $x\in \mc{M}$ and use geodesic normal coordinates around $x$. Since $\kappa$ is supported in $[0,1)$ and $h<\mathrm{inj}(\mc{M})$, every point in the support of $K_h(x,\cdot)$ can be written uniquely as
\begin{equation}
y=\exp_x(\xi),
\qquad
\xi\in B_h(0)\subset T_x\mc{M}.
\end{equation}
In these coordinates,
\begin{equation}
d_{\mc{M}}(x,\exp_x(\xi))=|\xi|,
\qquad
K_h(x,\exp_x(\xi))=h^{-d}\kappa(|\xi|/h),
\qquad
d\mu(\exp_x(\xi))=J_x(\xi)\,d\xi.
\end{equation}
Define
\begin{equation}
F_x^\alpha(\xi)\coloneq \alpha(x,\exp_x(\xi)),
\qquad
F_x^\beta(\xi)\coloneq \beta(x,\exp_x(\xi)).
\end{equation}
If $\xi\neq 0$, the minimizing geodesic from $x$ to $\exp_x(\xi)$ is $s\mapsto \exp_x(s\xi)$, $0\le s\le 1$, so
\begin{equation}
\label{eq:Fxalpha}
F_x^\alpha(\xi)
=
\int_0^1 \alpha_{\exp_x(s\xi)}\!\bigl((D\exp_x)_{s\xi}[\xi]\bigr)\,ds.
\end{equation}
The same formula holds for $F_x^\beta$. The integrands depend smoothly on $(x,s,\xi)$, so $F_x^\alpha$ and $F_x^\beta$ are smooth in $\xi$, uniformly in $x$ on the compact manifold. Also $F_x^\alpha(0)=0$, and differentiating \eqref{eq:Fxalpha} at $\xi=0$ gives $D_\xi F_x^\alpha(0)[\eta]=\alpha_x(\eta)$ because $(D\exp_x)_0=\mathrm{Id}_{T_x\mc{M}}$. Hence Taylor's theorem at $\xi=0$ yields
\begin{equation}
F_x^\alpha(\xi)
=
\alpha_x(\xi)+A_x^\alpha(\xi,\xi)+R_x^\alpha(\xi),
\end{equation}
where $A_x^\alpha$ is a symmetric bilinear form and $|R_x^\alpha(\xi)|\le C|\xi|^3$ uniformly in $x$. Similarly,
\begin{equation}
F_x^\beta(\xi)
=
\beta_x(\xi)+A_x^\beta(\xi,\xi)+R_x^\beta(\xi),
\qquad
|R_x^\beta(\xi)|\le C|\xi|^3.
\end{equation}
The Jacobian has the normal-coordinate expansion
\begin{equation}
J_x(\xi)=1+Q_x(\xi)+R_x^J(\xi),
\end{equation}
where $Q_x$ is quadratic and $|R_x^J(\xi)|\le C|\xi|^3$ uniformly in $x$.

Now multiply the three expansions. The leading term is the quadratic form $\alpha_x(\xi)\beta_x(\xi)$. The cubic terms come only from multiplying one linear term with one quadratic term:
\begin{equation}
P_{3,x}(\xi)\coloneq \alpha_x(\xi)A_x^\beta(\xi,\xi)+A_x^\alpha(\xi,\xi)\beta_x(\xi).
\end{equation}
Every remaining term is at least quartic in $\xi$: the factor $Q_x(\xi)$ is already quadratic, so even when multiplied by $\alpha_x(\xi)\beta_x(\xi)$ it contributes order $|\xi|^4$, and the remainders are cubic. Therefore
\begin{equation}
\label{eq:product-expansion-localization}
F_x^\alpha(\xi)F_x^\beta(\xi)J_x(\xi)
=
\alpha_x(\xi)\beta_x(\xi)+P_{3,x}(\xi)+R_x(\xi),
\end{equation}
where $P_{3,x}$ is a homogeneous cubic polynomial and $|R_x(\xi)|\le C|\xi|^4$.
The cubic term is odd in $\xi$, whereas the kernel $h^{-d}\kappa(|\xi|/h)$ is radial and the ball $B_h(0)$ is symmetric. Therefore
\begin{equation}
\int_{B_h(0)} h^{-d}\kappa(|\xi|/h)P_{3,x}(\xi)\,d\xi=0.
\end{equation}
Substituting \eqref{eq:product-expansion-localization} into the integral gives
\begin{align}
\int_{\mc{M}} K_h(x,y)\alpha(x,y)\beta(x,y)\,d\mu(y)
&=
\int_{B_h(0)} h^{-d}\kappa(|\xi|/h)\alpha_x(\xi)\beta_x(\xi)\,d\xi \\
&\qquad + O\!\left(\int_{B_h(0)} h^{-d}\kappa(|\xi|/h)|\xi|^4\,d\xi\right).
\end{align}
The remainder term is $O(h^4)$ uniformly in $x$, because
\begin{equation}
\int_{B_h(0)} h^{-d}\kappa(|\xi|/h)|\xi|^4\,d\xi
=
h^4\omega_{d-1}\int_0^1 \kappa(\rho)\rho^{d+3}\,d\rho.
\end{equation}
It remains to evaluate the principal term. Writing $\xi=t\theta$ with $0\le t<h$ and $\theta\in S_x\mc{M}$, we have $\alpha_x(t\theta)\beta_x(t\theta)=t^2\alpha_x(\theta)\beta_x(\theta)$, so
\begin{align}
\int_{B_h(0)} h^{-d}\kappa(|\xi|/h)\alpha_x(\xi)\beta_x(\xi)\,d\xi
&=
\int_0^h h^{-d}\kappa(t/h)t^{d+1}\,dt
\int_{S_x\mc{M}} \alpha_x(\theta)\beta_x(\theta)\,d\sigma_x(\theta) \notag \\
&=
h^2\left(\int_0^1 \kappa(\rho)\rho^{d+1}\,d\rho\right)
\int_{S_x\mc{M}} \alpha_x(\theta)\beta_x(\theta)\,d\sigma_x(\theta) \notag \\
&=
h^2\frac{m_2}{d}\langle \alpha_x,\beta_x\rangle_g,
\end{align}
where the last step uses the sphere-average identity and the definition of $m_2$. Combining this with the $O(h^4)$ remainder proves the lemma.
\end{proof}

\paragraph{Step 3: Integrate over the manifold and correct for sampling density}
Integrating the previous lemma over $x$ gives the first global formula.

\begin{corollary}[Continuum pairwise formula under uniform sampling]
Let $(\mc{M},g)$ be a smooth compact $d$-dimensional Riemannian manifold without boundary, let $\alpha,\beta\in \Omega^1(\mc{M})$ be smooth, and let $\kappa\colon [0,\infty)\to [0,\infty)$ be bounded and supported in $[0,1)$. Then
\begin{equation}
\langle \alpha,\beta\rangle_{L_2(\mu)}
=
\lim_{h\to 0}
\frac{d}{m_2h^2}
\int_{\mc{M}}\int_{\mc{M}}
K_h(x,y)\alpha(x,y)\beta(x,y)\,d\mu(y)\,d\mu(x).
\end{equation}
\end{corollary}
\begin{proof}
Integrating the short-range kernel localization lemma over $x$ gives
\begin{equation}
\int_{\mc{M}}\int_{\mc{M}}
K_h(x,y)\alpha(x,y)\beta(x,y)\,d\mu(y)\,d\mu(x)
=
h^2\frac{m_2}{d}\langle \alpha,\beta\rangle_{L_2(\mu)}+O(h^4),
\end{equation}
because the remainder is uniform in $x$ and $\mu(\mc{M})<\infty$. Dividing by $m_2h^2/d$ and letting $h\to 0$ proves the claim.
\end{proof}

In applications, the sample is usually not distributed according to the Riemannian volume. Instead, we observe a probability measure $\nu$ with smooth positive density $\pi$ relative to $\mu$:
\begin{equation}
d\nu=\pi\,d\mu.
\end{equation}
To remove this sampling bias, define the kernel density
\begin{equation}
q_h(x)\coloneq \int_{\mc{M}} K_h(x,z)\,d\nu(z).
\end{equation}
\begin{lemma}[Kernel-mass expansion]
\label{lem:kernel-mass-expansion}
Let $(\mc{M},g)$ be a smooth compact $d$-dimensional Riemannian manifold without boundary, let $\kappa\colon [0,\infty)\to [0,\infty)$ be bounded and supported in $[0,1)$, let $0<h<\mathrm{inj}(\mc{M})$, and let $d\nu=\pi\,d\mu$ with $\pi\in C^2(\mc{M})$. Then
\begin{equation}
\label{eq:kernel-mass-expansion}
q_h(x)=m_0\pi(x)+O(h^2)
\end{equation}
as $h\to 0$, uniformly in $x\in \mc{M}$.
\end{lemma}
\begin{proof}
Fix $x\in \mc{M}$ and use geodesic normal coordinates $z=\exp_x(\xi)$ on the support of $K_h(x,\cdot)$. Then
\begin{equation}
\label{eq:q-expansion-normal-coordinates}
q_h(x)
=
\int_{B_h(0)}
h^{-d}\kappa(|\xi|/h)\pi(\exp_x(\xi))J_x(\xi)\,d\xi,
\end{equation}
where $J_x(\xi)$ is the Jacobian of the Riemannian volume form. Since $\pi$ is $C^2$ and $J_x(\xi)=1+O(|\xi|^2)$ in normal coordinates, Taylor's theorem gives
\begin{equation}
\pi(\exp_x(\xi))J_x(\xi)
=
\pi(x)+L_x(\xi)+R_x(\xi),
\qquad
|R_x(\xi)|\le C|\xi|^2,
\end{equation}
where $L_x$ is linear in $\xi$ and the constant $C$ is uniform in $x$ by compactness. The kernel is radial and the ball $B_h(0)$ is symmetric, so the odd term integrates to $0$. Therefore
\begin{align}
q_h(x)
&=
\pi(x)\int_{B_h(0)} h^{-d}\kappa(|\xi|/h)\,d\xi
+ O\!\left(\int_{B_h(0)} h^{-d}\kappa(|\xi|/h)|\xi|^2\,d\xi\right) \notag \\
&=
\pi(x)m_0 + O(h^2),
\end{align}
where the first integral equals $m_0$ by the change of variables $\xi=h\rho\theta$, and the second is $O(h^2)$ for the same reason.
\end{proof}

Since $\pi$ is smooth and strictly positive on the compact manifold, it is bounded away from $0$. Hence \eqref{eq:kernel-mass-expansion} implies
\begin{equation}
\label{eq:density-ratio-expansion}
\frac{m_0\pi(x)}{q_h(x)}=1+O(h^2)
\end{equation}
uniformly in $x$.
Thus dividing by $q_h(x)q_h(y)$ cancels the leading effect of non-uniform sampling.

\begin{corollary}[Density-corrected continuum pairwise formula]
Let $(\mc{M},g)$ be a smooth compact $d$-dimensional Riemannian manifold without boundary, let $\alpha,\beta\in \Omega^1(\mc{M})$ be smooth, let $\kappa\colon [0,\infty)\to [0,\infty)$ be bounded and supported in $[0,1)$, and let $d\nu=\pi\,d\mu$ with $\pi>0$ smooth. Then
\begin{equation}
\label{eq:density-corrected-pairwise-formula}
\langle \alpha,\beta\rangle_{L_2(\mu)}
=
\lim_{h\to 0}
\frac{dm_0^2}{m_2h^2}
\int_{\mc{M}}\int_{\mc{M}}
\frac{K_h(x,y)}{q_h(x)q_h(y)}
\alpha(x,y)\beta(x,y)\,d\nu(y)\,d\nu(x).
\end{equation}
\end{corollary}
\begin{proof}
Using $d\nu=\pi\,d\mu$, we can rewrite the integral as
\begin{equation}
\label{eq:density-corrected-rewrite}
\frac{d}{m_2h^2}
\int_{\mc{M}}\int_{\mc{M}}
K_h(x,y)
\frac{m_0\pi(x)}{q_h(x)}
\frac{m_0\pi(y)}{q_h(y)}
\alpha(x,y)\beta(x,y)\,d\mu(y)\,d\mu(x).
\end{equation}
By \eqref{eq:density-ratio-expansion}, each factor $m_0\pi/q_h$ equals $1+O(h^2)$ uniformly, so the product of the two factors is also $1+O(h^2)$ uniformly. The short-range kernel localization lemma implies that
\begin{equation}
\label{eq:pairwise-integral-size}
\int_{\mc{M}}\int_{\mc{M}} K_h(x,y)\alpha(x,y)\beta(x,y)\,d\mu(y)\,d\mu(x)
=
h^2\frac{m_2}{d}\langle \alpha,\beta\rangle_{L_2(\mu)}+O(h^4).
\end{equation}
Substituting \eqref{eq:pairwise-integral-size} into \eqref{eq:density-corrected-rewrite} therefore yields
\begin{equation}
\frac{dm_0^2}{m_2h^2}
\int_{\mc{M}}\int_{\mc{M}}
\frac{K_h(x,y)}{q_h(x)q_h(y)}
\alpha(x,y)\beta(x,y)\,d\nu(y)\,d\nu(x)
=
\langle \alpha,\beta\rangle_{L_2(\mu)}+O(h^2),
\end{equation}
and taking $h\to 0$ proves the corollary.
\end{proof}

\paragraph{Step 4: Discretize the pairwise formula}
We are able to approximate \cref{eq:density-corrected-pairwise-formula} using a discrete point cloud: Let $E_1$ be a choice of orientation of the edges of $\mc{S}_\varepsilon$. For an edge $e=(x_i,x_j)\in E_1$, define
\begin{equation}
q_i\coloneq \sum_{\ell\neq i} K_h(x_i,x_\ell),
\end{equation}
and
\begin{equation}
w_e\coloneq
\begin{cases}
\dfrac{2dm_0^2}{m_2h^2}\dfrac{K_h(x_i,x_j)}{q_iq_j}, & \text{if } K_h(x_i,x_j)>0,\\[0.8ex]
0, & \text{if } K_h(x_i,x_j)=0.
\end{cases}
\end{equation}
On the support of $K_h$, the corresponding terms in the sums defining $q_i$ and $q_j$ are strictly positive, so the denominator is well-defined there.
We then define the cochain inner product by
\begin{equation}
\langle c,c'\rangle_M
\coloneq
\sum_{e\in E_1} w_e\,c(e)c'(e),
\qquad
M=\diag(w_e)_{e\in E_1}.
\end{equation}
This is the discrete analogue of the density-corrected continuum bilinear form. The factor $2$ is part of the definition because the continuum double integral counts both ordered pairs $(x_i,x_j)$ and $(x_j,x_i)$, whereas the discrete sum uses only one chosen orientation of each undirected edge. The prefactor $dm_0^2/(m_2h^2)$ is the continuum normalization from \cref{eq:density-corrected-pairwise-formula}, so the same matrix $M$ can now be used both in the algorithm and in the convergence theorem below. In practice, $h$ should be chosen on the same scale as the simplicial parameter $\varepsilon$, although we do not require $h=\varepsilon$. If one only cares about the later optimisation problem, then any positive global scalar multiple of $M$ leads to exactly the same minimisers.

We note that we do not divide each edge contribution by its length: A $1$-cochain stores the edge integral $c_\alpha(x_i,x_j)=\int_{x_i}^{x_j}\alpha$, and for a short edge of length $r$ we have $c_\alpha(x_i,x_j)=r\,\alpha_{x_i}(\theta)+O(r^2)$ with $\theta$ the edge direction. Hence the product $c_\alpha c_\beta$ already carries a factor $r^2$. Dividing by $r^2$ would therefore convert the cochain into an estimator of the \emph{directional values} of the forms, which is the more natural normalization when discretizing vector fields or pointwise covectors. Here, however, the basic object is the cochain itself, because cohomology and the harmonic-representative problem are formulated in terms of edge integrals.
The present estimator is therefore designed for discrete $1$-cochains, not for discrete vector fields.

\paragraph{Convergence guarantee}
In the previous steps, we showed that $M$ is the correct discretization of the $L_2$ inner product on $1$-forms. In the following part, we will now show that the inner product induced by $M$ will converge to the $L_2$ inner product in a precise sense, when increasing sample size and decreasing bandwidth with the correct rates:

Fix smooth $\alpha,\beta\in \Omega^1(\mc{M})$. On each oriented edge $e=(x_i,x_j)$, let
\begin{equation}
\label{eq:def-sampled-cochains}
c_\alpha(e)\coloneq \alpha(x_i,x_j),
\qquad
c_\beta(e)\coloneq \beta(x_i,x_j).
\end{equation}
The continuum comparison functional is
\begin{equation}
I_h(\alpha,\beta)
\coloneq
\frac{dm_0^2}{m_2h^2}
\int_{\mc{M}}\int_{\mc{M}}
\frac{K_h(x,y)}{q_h(x)q_h(y)}
\alpha(x,y)\beta(x,y)\,d\nu(y)\,d\nu(x)
\end{equation}
and the normalized discrete quantity from Step~4 is
\begin{equation}
\label{eq:discrete-inner-product-expanded}
\langle c_\alpha,c_\beta\rangle_M
=
\frac{2dm_0^2}{m_2h^2}
\sum_{\substack{e=(x_i,x_j)\in E_1\\ K_h(x_i,x_j)>0}}
\frac{K_h(x_i,x_j)}{q_iq_j}\,c_\alpha(e)c_\beta(e).
\end{equation}

The argument now becomes completely linear:
\begin{equation}
\label{eq:error-decomposition}
\bigl|\langle c_\alpha,c_\beta\rangle_M-\langle \alpha,\beta\rangle_{L_2(\mu)}\bigr|
\le
\bigl|\langle c_\alpha,c_\beta\rangle_M-I_h(\alpha,\beta)\bigr|
+
\bigl|I_h(\alpha,\beta)-\langle \alpha,\beta\rangle_{L_2(\mu)}\bigr|.
\end{equation}
The second term is the deterministic approximation error from replacing infinitesimal information by a kernel average. The first term is the random sampling error from replacing a double integral by a finite edge sum.

\begin{proposition}[Deterministic bias]
Let $(\mc{M},g)$ be a smooth compact $d$-dimensional Riemannian manifold without boundary, let $\alpha,\beta\in \Omega^1(\mc{M})$ be smooth, let $\kappa\colon [0,\infty)\to [0,\infty)$ be bounded and supported in $[0,1)$, and let $d\nu=\pi\,d\mu$ with $\pi\in C^2(\mc{M})$ strictly positive. Then
\begin{equation}
I_h(\alpha,\beta)
=
\langle \alpha,\beta\rangle_{L_2(\mu)}+O(h^2)
\end{equation}
as $h\to 0$.
\end{proposition}
\begin{proof}
By the proof of the density-corrected continuum pairwise formula, see in particular \eqref{eq:density-corrected-rewrite}, \eqref{eq:density-ratio-expansion}, and \eqref{eq:pairwise-integral-size}, we have
\begin{equation}
I_h(\alpha,\beta)
=
\langle \alpha,\beta\rangle_{L_2(\mu)}+O(h^2).
\end{equation}
The geometric reason for the $h^2$ bias is the same as in the localization lemma: after multiplying the Taylor expansions, the cubic term is odd and integrates to $0$ against the radial kernel, so the first surviving remainder is quartic.
\end{proof}

The stochastic term looks more complicated, but its meaning is simple. For fixed $h$, only pairs with distance $O(h)$ contribute. Each point therefore interacts with about $N h^d$ neighbours, which is why the natural fluctuation size is the square-root law
\begin{equation}
\sqrt{\frac{\log N}{N h^d}}.
\end{equation}
The logarithm appears because we need uniform control of the density estimates entering the denominator.

\begin{proposition}[Sampling error]
Assume that
\begin{enumerate}
\item $(\mc{M},g)$ is a smooth compact $d$-dimensional Riemannian manifold without boundary;
\item the sample $X=\{x_1,\dots,x_N\}$ is drawn i.i.d. from a measure $\nu=\pi\,\mu$, where $\pi\in C^2(\mc{M})$ is strictly positive;
\item the kernel profile $\kappa$ is bounded, supported in $[0,1)$, and satisfies $m_2>0$;
\item $h=h_N\to 0$, $N h_N^d/\log N\to \infty$, and $h_N<\mathrm{inj}(\mc{M})$ for all sufficiently large $N$;
\item the $1$-skeleton of the chosen simplicial complex on $X$ contains every unordered pair $\{x_i,x_j\}$ for which $K_{h_N}(x_i,x_j)\neq 0$.
\end{enumerate}
Let $M=M_{N,h_N}$ be the diagonal matrix defined in Step~4 from this sample and bandwidth $h_N$.
Then, for every fixed smooth $\alpha,\beta\in \Omega^1(\mc{M})$, the sampled cochains from \eqref{eq:def-sampled-cochains} satisfy
\begin{equation}
\langle c_\alpha,c_\beta\rangle_M-I_{h_N}(\alpha,\beta)
=
O_{\mathbb P}\!\left(\sqrt{\frac{\log N}{N h_N^d}}\right).
\end{equation}
\end{proposition}
\begin{proof}
For convenience, write
\begin{equation}
r_i\coloneq
\begin{cases}
\dfrac{N}{q_i}, & \text{if } q_i>0,\\[0.8ex]
0, & \text{if } q_i=0,
\end{cases}
\qquad
s_i\coloneq \frac{1}{q_h(x_i)},
\qquad
\Delta_{N,h}\coloneq \sup_{1\le i\le N}|r_i-s_i|.
\end{equation}
The proof has two steps.

\paragraph{Step 1: Replace the empirical reciprocals}
Fix $i$. Conditional on $x_i$, the random variables
\begin{equation}
Y_{i\ell}\coloneq K_h(x_i,x_\ell),
\qquad
\ell\neq i,
\end{equation}
are i.i.d.\ by assumption \textup{(2)}, bounded by $\|\kappa\|_\infty h^{-d}$ by assumption \textup{(3)}, and satisfy
\begin{equation}
\mathbb E[Y_{i\ell}\mid x_i]=q_h(x_i),
\qquad
\mathbb E[Y_{i\ell}^2\mid x_i]\le \|\kappa\|_\infty h^{-d}q_h(x_i)\le Ch^{-d}.
\end{equation}
Hence Bernstein's inequality yields constants $c,C>0$, independent of $i$, such that for every $0<t\le 1$,
\begin{equation}
\label{eq:conditional-bernstein}
\mathbb P\!\left(
\left|
\frac{1}{N-1}\sum_{\ell\neq i}Y_{i\ell}-q_h(x_i)
\right|>t \,\middle|\, x_i
\right)
\le
2\exp\!\left(-cNh^dt^2\right).
\end{equation}
Taking expectations and applying a union bound over $i$ gives
\begin{equation}
\label{eq:q-over-n-minus-one}
\sup_{1\le i\le N}
\left|
\frac{q_i}{N-1}-q_h(x_i)
\right|
=
O_{\mathbb P}\!\left(\sqrt{\frac{\log N}{N h^d}}\right).
\end{equation}
Since $q_i/N=\frac{N-1}{N}\cdot q_i/(N-1)$ and $q_h$ is uniformly bounded, the difference between $q_i/N$ and $q_i/(N-1)$ is $O_{\mathbb P}(1/N)$ uniformly in $i$, hence
\begin{equation}
\label{eq:q-over-n}
\sup_{1\le i\le N}
\left|
\frac{q_i}{N}-q_h(x_i)
\right|
=
O_{\mathbb P}\!\left(\sqrt{\frac{\log N}{N h^d}}\right).
\end{equation}

By \eqref{eq:kernel-mass-expansion}, whose proof uses the geometric hypotheses in assumption \textup{(1)}, and by the lower bound on $\pi$ from assumption \textup{(2)}, there exists $c_0>0$ such that $q_h(x)\ge c_0$ for all $x$ and all sufficiently small $h$. Since assumption \textup{(4)} gives $\sqrt{\log N/(N h^d)}\to 0$, \eqref{eq:q-over-n} implies that the event
\begin{equation}
\sup_i\left|\frac{q_i}{N}-q_h(x_i)\right|\le \frac{c_0}{2},
\end{equation}
has probability tending to $1$. On this event we also have $q_i/N\ge c_0/2$ for every $i$. Therefore,
\begin{equation}
\left|r_i-s_i\right|
=
\left|
\frac{1}{q_i/N}-\frac{1}{q_h(x_i)}
\right|
\le
\frac{2}{c_0^2}
\left|
\frac{q_i}{N}-q_h(x_i)
\right|,
\end{equation}
and \eqref{eq:q-over-n} implies
\begin{equation}
\label{eq:reciprocal-rate}
\Delta_{N,h}
=
O_{\mathbb P}\!\left(\sqrt{\frac{\log N}{N h^d}}\right).
\end{equation}

Define
\begin{equation}
\label{eq:def-I-tilde}
\widetilde I_{N,h}(c_\alpha,c_\beta)
\coloneq
\frac{2dm_0^2}{m_2h^2}
\sum_{e=(x_i,x_j)\in E_1}
\frac{K_h(x_i,x_j)}{N^2q_h(x_i)q_h(x_j)}c_\alpha(e)c_\beta(e).
\end{equation}
This is well-defined because assumption \textup{(3)} gives $m_2>0$.
Because $r_i$ and $s_i$ are uniformly $O_{\mathbb P}(1)$, the product difference satisfies
\begin{equation}
\sup_{i,j}|r_ir_j-s_is_j|
=
O_{\mathbb P}(\Delta_{N,h}).
\end{equation}
Substituting this bound into \eqref{eq:discrete-inner-product-expanded} and \eqref{eq:def-I-tilde} yields
\begin{align}
\label{eq:difference-I-Itilde}
\left|\langle c_\alpha,c_\beta\rangle_M-\widetilde I_{N,h}(c_\alpha,c_\beta)\right|
&\le
C\Delta_{N,h}\,
\frac{1}{N^2h^2}
\sum_{e=(x_i,x_j)\in E_1}
K_h(x_i,x_j)\,|c_\alpha(e)c_\beta(e)|.
\end{align}
Because assumption \textup{(3)} says that $\kappa$ is supported in $[0,1)$ and assumption \textup{(4)} gives $h<\mathrm{inj}(\mc{M})$ for all sufficiently large $N$, every contributing pair satisfies $d_{\mc{M}}(x_i,x_j)<h$ and the geodesic integrals defining $c_\alpha(e)$ and $c_\beta(e)$ are well-defined. Since $\mc{M}$ is compact by assumption \textup{(1)}, smoothness of the fixed forms $\alpha,\beta$ gives the uniform bound
\begin{equation}
\label{eq:cochain-size}
|c_\alpha(e)|+|c_\beta(e)|\le Ch.
\end{equation}
Therefore
\begin{equation}
\frac{1}{N^2h^2}
\sum_{e=(x_i,x_j)\in E_1} K_h(x_i,x_j)\,|c_\alpha(e)c_\beta(e)|
\le
\frac{C}{N^2}
\sum_{1\le i\neq j\le N} K_h(x_i,x_j).
\end{equation}
The expectation of this quantity is
\begin{equation}
\frac{1}{N^2}\sum_{i\neq j}\mathbb E[K_h(X_i,X_j)]
=
\frac{N-1}{N}\,\mathbb E[q_h(X_1)]
\le C,
\end{equation}
so Markov's inequality gives
\begin{equation}
\frac{1}{N^2h^2}
\sum_{e\in E_1}K_h(x_i,x_j)\,|c_\alpha(e)c_\beta(e)|
=
O_{\mathbb P}(1).
\end{equation}
Together with \eqref{eq:reciprocal-rate}, this proves
\begin{equation}
\label{eq:reciprocal-replacement-error}
\langle c_\alpha,c_\beta\rangle_M-\widetilde I_{N,h}(c_\alpha,c_\beta)
=
O_{\mathbb P}\!\left(\sqrt{\frac{\log N}{N h^d}}\right).
\end{equation}

\paragraph{Step 2: Compare the pairwise average with the continuum expectation}
After the reciprocal replacement, the discrete sum becomes a pairwise average of the kernel
\begin{equation}
\label{eq:def-Phi-h}
\Phi_h(x,y)
\coloneq
\frac{dm_0^2}{m_2h^2}
\frac{K_h(x,y)}{q_h(x)q_h(y)}
\alpha(x,y)\beta(x,y).
\end{equation}
\begin{equation}
\label{eq:def-U-N-h}
U_{N,h}
\coloneq
\frac{2}{N(N-1)}
\sum_{1\le i<j\le N}\Phi_h(x_i,x_j).
\end{equation}
Because $\alpha(y,x)=-\alpha(x,y)$ and $\beta(y,x)=-\beta(x,y)$, the product $\alpha(x,y)\beta(x,y)$ is symmetric in $(x,y)$, hence so is $\Phi_h$. Assumption \textup{(5)} therefore implies
\begin{equation}
\label{eq:Itilde-vs-U}
\widetilde I_{N,h}(c_\alpha,c_\beta)
=
\frac{N-1}{N}\,U_{N,h}.
\end{equation}

On the support of $K_h$, smoothness gives $|\alpha(x,y)|+|\beta(x,y)|\le Ch$, hence
\begin{equation}
\label{eq:phi-moments}
\|\Phi_h\|_\infty\le C h^{-d},
\qquad
\mathbb E[\Phi_h(X_1,X_2)^2]\le C h^{-d}.
\end{equation}
Here the support condition in assumption \textup{(3)} and the injectivity-radius condition in assumption \textup{(4)} ensure that $\alpha(x,y)$ and $\beta(x,y)$ are evaluated only on geodesic pairs with $d_{\mc{M}}(x,y)<h$, and compactness from assumption \textup{(1)} gives the uniform constants. The second bound follows from the first because assumption \textup{(3)} also gives $K_h^2\le \|\kappa\|_\infty h^{-d}K_h$.

Set
\begin{equation}
g_h(x)\coloneq \mathbb E[\Phi_h(x,X_2)]-I_h(\alpha,\beta),
\qquad
H_h(x,y)\coloneq \Phi_h(x,y)-I_h(\alpha,\beta)-g_h(x)-g_h(y).
\end{equation}
Then $\mathbb E[g_h(X_1)]=0$ and
\begin{equation}
\mathbb E[H_h(x,X_2)]=\mathbb E[H_h(X_1,y)]=0.
\end{equation}
Because the sample points are i.i.d.\ by assumption \textup{(2)}, the Hoeffding decomposition gives
\begin{equation}
\label{eq:hoeffding}
U_{N,h}-I_h(\alpha,\beta)
=
\frac{2}{N}\sum_{i=1}^N g_h(x_i)
+
\frac{2}{N(N-1)}
\sum_{1\le i<j\le N} H_h(x_i,x_j).
\end{equation}
Moreover,
\begin{equation}
|g_h(x)|
\le
\int_{\mc{M}}|\Phi_h(x,y)|\,d\nu(y)+|I_h(\alpha,\beta)|
\le C,
\end{equation}
because $|\alpha(x,y)\beta(x,y)|\le Ch^2$ on the support of $K_h$ and $\int K_h(x,y)\,d\nu(y)=q_h(x)=O(1)$ uniformly. Hence
\begin{equation}
\label{eq:linear-part-variance}
\mathbb E\!\left[
\left(
\frac{2}{N}\sum_{i=1}^N g_h(X_i)
\right)^2
\right]
=
O\!\left(\frac{1}{N}\right).
\end{equation}

For the canonical part, mixed second moments vanish unless the two unordered pairs coincide: if $\{i,j\}\neq \{k,\ell\}$, then either the pairs are disjoint, in which case independence gives zero, or they share exactly one index, in which case conditioning on the shared variable and using the canonical property again gives zero. Therefore
\begin{equation}
\label{eq:canonical-second-moment}
\mathbb E\!\left[
\left(
\frac{2}{N(N-1)}
\sum_{1\le i<j\le N} H_h(X_i,X_j)
\right)^2
\right]
\le
\frac{C}{N^2}\,\mathbb E[\Phi_h(X_1,X_2)^2]
=
O\!\left(\frac{1}{N^2h^d}\right).
\end{equation}
By Chebyshev's inequality, \eqref{eq:linear-part-variance}, and \eqref{eq:canonical-second-moment},
\begin{equation}
\label{eq:u-stat-rate}
U_{N,h}-I_h(\alpha,\beta)
=
O_{\mathbb P}\!\left(\frac{1}{\sqrt{N}}\right)
+
O_{\mathbb P}\!\left(\frac{1}{N h^{d/2}}\right)
=
O_{\mathbb P}\!\left(\sqrt{\frac{1}{N h^d}}\right).
\end{equation}
The last equality uses $h\to 0$, so $h^d\le 1$ for all sufficiently large $N$ and therefore $N^{-1/2}\le (Nh^d)^{-1/2}$.

Since \eqref{eq:Itilde-vs-U} implies
\begin{equation}
\widetilde I_{N,h}(c_\alpha,c_\beta)-U_{N,h}
=
O_{\mathbb P}\!\left(\frac{1}{N}\right),
\end{equation}
because $U_{N,h}=I_h(\alpha,\beta)+O_{\mathbb P}((Nh^d)^{-1/2})=O_{\mathbb P}(1)$, we obtain
\begin{equation}
\label{eq:Itilde-minus-Ih}
\widetilde I_{N,h}(c_\alpha,c_\beta)-I_h(\alpha,\beta)
=
O_{\mathbb P}\!\left(\sqrt{\frac{1}{N h^d}}\right).
\end{equation}
Combining \eqref{eq:reciprocal-replacement-error} and \eqref{eq:Itilde-minus-Ih} proves the claim.
\end{proof}

\begin{theorem}[Consistency of the discrete inner product]
\label{thm:discrete-inner-product-consistency}
Assume that
\begin{enumerate}
\item $(\mc{M},g)$ is a smooth compact $d$-dimensional Riemannian manifold without boundary;
\item the sample $X=\{x_1,\dots,x_N\}$ is drawn i.i.d. from a measure $\nu=\pi\,\mu$, where $\pi\in C^2(\mc{M})$ is strictly positive;
\item the kernel profile $\kappa$ is bounded, supported in $[0,1)$, and satisfies $m_2>0$;
\item $h=h_N\to 0$, $N h_N^d/\log N\to \infty$, and $h_N<\mathrm{inj}(\mc{M})$ for all sufficiently large $N$;
\item the $1$-skeleton of the chosen simplicial complex on $X$ contains every unordered pair $\{x_i,x_j\}$ for which $K_{h_N}(x_i,x_j)\neq 0$.
\end{enumerate}
Let $M=M_{N,h_N}$ be the diagonal matrix defined in Step~4 from this sample and bandwidth $h_N$.
Then, for every fixed smooth $\alpha,\beta\in \Omega^1(\mc{M})$, the sampled cochains from \eqref{eq:def-sampled-cochains} satisfy
\begin{equation}
\label{eq:discrete-inner-product-consistency-estimate}
\langle c_\alpha,c_\beta\rangle_M
=
\langle \alpha,\beta\rangle_{L_2(\mu)}
+ O(h_N^2)
+ O_{\mathbb P}\!\left(\sqrt{\frac{\log N}{N h_N^d}}\right).
\end{equation}
In particular,
\begin{equation}
\langle c_\alpha,c_\beta\rangle_M
\xrightarrow[N\to\infty]{\mathbb{P}}
\langle \alpha,\beta\rangle_{L_2(\mu)}.
\end{equation}
\end{theorem}
\begin{proof}
Apply the decomposition \eqref{eq:error-decomposition}. The deterministic-bias proposition gives
\begin{equation}
I_{h_N}(\alpha,\beta)-\langle \alpha,\beta\rangle_{L_2(\mu)}=O(h_N^2),
\end{equation}
while the sampling-error proposition gives
\begin{equation}
\langle c_\alpha,c_\beta\rangle_M-I_{h_N}(\alpha,\beta)
=
O_{\mathbb P}\!\left(\sqrt{\frac{\log N}{N h_N^d}}\right).
\end{equation}
Adding the two bounds proves \eqref{eq:discrete-inner-product-consistency-estimate}. The hypotheses $h_N\to 0$ and $N h_N^d/\log N\to \infty$ force both error terms to vanish, hence
\begin{equation}
\langle c_\alpha,c_\beta\rangle_M\xrightarrow[N\to\infty]{\mathbb{P}}\langle \alpha,\beta\rangle_{L_2(\mu)}.
\end{equation}
\end{proof}

We can interpret this in the following way:
\begin{enumerate}
\item the term $h_N^2$ is the geometric bias from replacing an infinitesimal inner product by an average over a ball of radius $h_N$;
\item the term $\sqrt{\log N/(N h_N^d)}$ is the finite-sample error from estimating that local average using only finitely many nearby pairs.
\end{enumerate}

\begin{corollary}[Balanced bandwidth]
Under the assumptions of \cref{thm:discrete-inner-product-consistency}, if
\begin{equation}
h_N\asymp \left(\frac{\log N}{N}\right)^{1/(d+4)},
\end{equation}
then, for every fixed smooth $\alpha,\beta\in \Omega^1(\mc{M})$,
\begin{equation}
\langle c_\alpha,c_\beta\rangle_M
=
\langle \alpha,\beta\rangle_{L_2(\mu)}
+ O_{\mathbb P}\!\left(\left(\frac{\log N}{N}\right)^{2/(d+4)}\right).
\end{equation}
\end{corollary}
\begin{proof}
For this choice of bandwidth,
\begin{equation}
h_N^2\asymp \left(\frac{\log N}{N}\right)^{2/(d+4)}
\end{equation}
and
\begin{equation}
\sqrt{\frac{\log N}{N h_N^d}}
\asymp
\sqrt{
\frac{\log N}{N}
\left(\frac{N}{\log N}\right)^{d/(d+4)}
}
=
\left(\frac{\log N}{N}\right)^{2/(d+4)}.
\end{equation}
Substituting these two relations into \cref{thm:discrete-inner-product-consistency} gives the claim.
\end{proof}

The theorem establishes consistency of the bilinear form on cochains obtained by sampling fixed smooth forms. It does not yet imply convergence of the discrete harmonic representatives produced by the constrained minimisation problem on $C^1(\mc{S}_\varepsilon;\R)$.

In \cref{sec:KernelRemarks}, we will discuss why the choice of kernel normalisation by different powers of $r$ is not expected to improve the convergence rate.

\section{Remarks on the choice of kernel, density correction, and normalization}
\label{sec:KernelRemarks}
One might ask whether replacing
\begin{equation}
w_e=\frac{2dm_0^2}{m_2h^2}\frac{K_h(x_i,x_j)}{q_iq_j}
\end{equation}
by a more general family
\begin{equation}
w_e^{(a)}\coloneq \frac{2dm_0^{2a}}{m_2h^2}\frac{K_h(x_i,x_j)}{q_i^a q_j^a},
\qquad a\in\R,
\end{equation}
could improve the convergence.
The corresponding continuum functional is
\begin{equation}
I_h^{(a)}(\alpha,\beta)
\coloneq
\frac{d m_0^{2a}}{m_2 h^2}
\int_{\mc{M}}\int_{\mc{M}}
\frac{K_h(x,y)}{q_h(x)^a q_h(y)^a}
\alpha(x,y)\beta(x,y)\,d\nu(y)\,d\nu(x).
\end{equation}
Using $q_h(x)=m_0\pi(x)+O(h^2)$ and $d\nu=\pi\,d\mu$, we obtain
\begin{equation}
I_h^{(a)}(\alpha,\beta)
=
\int_{\mc{M}} \pi(x)^{2-2a}\langle \alpha_x,\beta_x\rangle_g\,d\mu(x)
+ O(h^2).
\end{equation}
Therefore $a=1$ is the unique exponent in this family that converges to the geometric inner product $\langle \alpha,\beta\rangle_{L_2(\mu)}$.

\end{document}